\documentclass[11pt,a4paper]{article}
\usepackage[english]{babel}
\usepackage{amsmath}
\usepackage{amssymb}
\usepackage{amsthm}

\newtheorem{theorem}{Theorem}[section]

\newtheorem{corollary}[theorem]{Corollary}
\newtheorem{lemma}[theorem]{Lemma}
\newtheorem{proposition}[theorem]{Proposition}
\newtheorem{definition}[theorem]{Definition}
\newtheorem{remark}[theorem]{Remark}
\newtheorem{example}[theorem]{Example}

\newtheorem{notation}[theorem]{Notation}

\def\cK{\mathcal{K}}

\newcounter{c}
\newcounter{tmpabcd}

\newcounter{tmpabcdmine}

\newcounter{tmpnum}

\newcounter{tmprome}

\newcommand{\T}{\ensuremath{\mathcal{T}}}

\newcommand{\car}{{\rm char\,}}

\newcommand{\mcc}{\mathbb{C}}

\newcommand{\mpp}{\mathbb{P}}

\newcommand{\mnn}{\mathbb{N}}

\newcommand{\mrr}{\mathbb{R}}

\newcommand{\mzz}{\mathbb{Z}}

\newcommand{\mqq}{\mathbb{Q}}

\newcommand{\diff}[1]{\frac{\partial}{\partial #1}}

\newcommand{\eqdef}{\ensuremath{\stackrel{\mathrm{def}}{=}}}

\newcommand{\rk}{\operatorname{rk}}

\newcommand{\Dist}{\ensuremath{{\rm Dist}}}
\newcommand{\dist}{\ensuremath{{\rm dist}}}

\newcommand{\dd}{\ensuremath{{\rm deg}}}

\renewcommand{\b}[1]{{{#1}}}

\newcommand{\hidden}[1]{}

\newcommand{\e}{{\rm e}}

\newcommand{\V}{\ensuremath{\mathcal{V}}}
\newcommand{\I}{\ensuremath{\mathcal{I}}}
\newcommand{\Zeros}{\ensuremath{\mathcal{Z}}}

\newcommand{\A}{\ensuremath{\mathcal{A}}}
\newcommand{\AnneauDePolynomes}{\A}

\newcommand{\Z}{\ensuremath{\mathcal{Z}}}

\newcommand{\Ciso}{C_{\mbox{iso}}}

\newcommand{\kk}{\ensuremath{\Bbbk}}

\newcommand{\idp}{\ensuremath{\mathcal{P}}}
\newcommand{\idq}{\ensuremath{\mathcal{Q}}}
\newcommand{\idpf}{\ensuremath{\mathcal{P}_{\ul{f}}}}

\newcommand{\Spec}{\ensuremath{{\rm Spec}}}
\newcommand{\Ass}{\ensuremath{{\rm Ass}}}

\newcommand{\ul}[1]{\underline{#1}}
\newcommand{\ull}[1]{\underline{\b{#1}}}
\newcommand{\ullt}[1]{\underline{\b{#1}}}
\newcommand{\ol}[1]{\overline{#1}}

\newcommand{\bt}[1]{{\b{\tilde{#1}}}}

\newcommand{\ord}{\ensuremath{{\rm ord}}}
\newcommand{\ordz}{\ensuremath{{\rm ord_{\b{z}=0}}}}
\newcommand{\Ord}{\ensuremath{{\rm Ord}}}
\newcommand{\Ordz}{\ensuremath{{\rm Ord_{z=0}}}}
\newcommand{\trdeg}{\ensuremath{{\rm tr.deg.}}}

\newcommand{\Talg}{\ensuremath{\ul{\T}^*}}
\newcommand{\irrT}{\ensuremath{\irr \, \T}}

\newcommand{\MM}{\mathcal{\tilde{M}}}

\newcommand{\irr}{{\rm irr}}
\newcommand{\rg}{{\rm rk}}

\newcommand{\eq}{{\rm eq}}

\begin{document}

\markboth{Evgeniy Zorin}
{Zero order estimates for analytic functions}

\title{MULTIPLICITY ESTIMATES FOR ALGEBRAICALLY DEPENDENT ANALYTIC FUNCTIONS
}

\author{EVGENIY ZORIN} 
\date{}
\maketitle



\date{}

\maketitle

\begin{abstract}
We prove a new general multiplicity estimate applicable to sets of functions without any assumption on algebraic independence. The multiplicity estimates are commonly used in determining measures of algebraic independence of values of functions, for instance within the context of Mahler's method. For this reason, our result provides an important tool for the proofs of algebraic independence of complex numbers. At the same time, these estimates can be considered as a measure of algebraic independence of functions themselves. Hence our result provides, under some conditions, the measure of algebraic independence of elements in ${\bf F}_q[[T]]$, where ${\bf F}_q$ denotes a finite field. 
\end{abstract}





\section{Introduction}

Let $\kk$ be any field and let
\begin{equation} \label{intro_functions_f}
f_1(z),\dots,f_n(z)\in\kk[[z]]
\end{equation}
be a collection of formal power series with coefficients in $\kk$. 
In this article we are interested in so-called \emph{multiplicity estimates}, referred to as uniform upper bounds for
$$
\ordz P(z,f_1(z),\dots,f_n(z)),
$$
the order of vanishing of $ P(z,f_1(z),\dots,f_n(z))$ at $z=0$, where $P$ is a polynomial $P\in\kk[Z,X_1,\dots,X_n]$ in $n+1$ variables. Naturally, the upper bounds should depend on the degree of $P$, since, for example, $\ordz P_N(z,f_1(z))>N$ when $f_1(z)=\sum_{i=0}^{\infty}a_nz^n$ and $P_N(Z,X_1)=X_1-\sum_{i=0}^{N}a_nZ^n$. In applications it is often desirable to have upper bounds of the form
\begin{equation} \label{def_LM_a}
	\ordz P(z,f_1(z),\dots,f_n(z))<F(\deg_{z}(P),\deg_{\ul{X}}(P)),
\end{equation}
where $F$ is independent of $P$ and the degrees in $z$ and in $X_1,\dots,X_n$ are separated. We say that  $f_1(z),\dots,f_n(z)$ verify \emph{a multiplicity estimate/lemma} (with respect to $F$) if~\eqref{def_LM_a} holds for all $P\in\kk[Z,X_1,\dots,X_n]$ such that $ P(z,f_1(z),\dots,f_n(z))\ne 0$.


Multiplicity estimates form an important tool in transcendental number theory for establishing the measure of algebraic independence of, for example, complex numbers. The multiplicity estimates can also be considered as measures of algebraic independence of functions (or formal power series). For more detailed explanation we refer the reader to~\cite{Bertrand1985,Bertrand2007a,Bertrand2007b,Bertrand2007}.

The central theme in transcendental number theory is to determine whether or not the values $f_1(\alpha),\dots,f_n(\alpha)$, of a given set of analytic functions given by~\eqref{intro_functions_f},
are algebraically independent at algebraic points $\alpha$.
In the next few pages we explain why it is natural to apply the estimates~\eqref{def_LM_a} within this theme.
To begin with, recall the famous Lindemann-Weierstrass Theorem:
\begin{theorem}[(Lindemann-Weierstrass)]
Let $n\geq 2$ be an integer and $\alpha_1,\dots,\alpha_n\in\mcc$ be algebraic over $\mqq$. Suppose that $\exp(\alpha_1z),\dots,\exp(\alpha_nz)$ are algebraically independent over $\mqq$ functions of a complex variable $z$. Then for any $\beta\in\ol{\mqq}^*$ the numbers $\exp(\alpha_1\beta),\dots,\exp(\alpha_n\beta)$ are algebraically independent over $\mqq$.
\end{theorem}
\begin{remark}
Usually this theorem is stated with the hypothesis that $\alpha_1,\dots,\alpha_n$ are linearly independent over $\mqq$, instead of algebraic independence of functions $\exp(\alpha_1z),\dots,\exp(\alpha_nz)$. It is an easy exercise to verify that these conditions are equivalent.
\end{remark}
Later this result was generalized to the much broader class of so-called $E$-functions (`$E$' here is for ``exponential'', as the definition of this class captures some important properties of the exponential function). We refer the reader to~\cite{Sh1959} for the definition of this class. 
This line of research was initiated by Siegel with an impressive development via a number of  \emph{tours de force}~\cite{Sh1959,NS1996,Andre2000,Beukers2006}, crowned by the following qualitatively best possible result.
\begin{theorem}[(Nesterenko-Shidlovsky, \cite{NS1996})] \label{intro_theo_NS}
Let $f_1(z),\dots,f_n(z)$ be a set of $E$-functions that form  a solution of the system of first-order differential equations
\begin{equation} \label{NS_theo_h_system}
\frac{d}{dz}\begin{pmatrix}f_1(z)\\\vdots\\ f_n(z)\end{pmatrix}=A(z)\begin{pmatrix}f_1(z)\\\vdots\\ f_n(z)\end{pmatrix},
\end{equation}
where $A$ is an $n\times n$-matrix with entries in $\ol{\mqq}(z)$. Denote by $T(z)$ the common denominator
of the entries of $A$. Then, for any $\alpha\in\ol{\mqq}$ such that $\alpha T(\alpha)\ne 0$,
\begin{equation} \label{intro_trdegf_eq_trdegalpha}
\trdeg_{\mqq}\mqq\left(f_1(\alpha),\dots,f_n(\alpha)\right)=\trdeg_{\mcc(z)}\mcc(z)\left(f_1(z),\dots,f_n(z)\right).
\end{equation}
Moreover, there exists a finite set $S$ such that for all $\alpha\in\ol{\mqq}$, $\alpha\not\in S$ the following holds. For any homogeneous polynomial $P\in\ol{\mqq}[X_1,\dots,X_n]$ with $P(f_1(\alpha),\dots,f_n(\alpha))=0$ there exists $Q\in\ol{\mqq}[z,X_1,\dots,X_n]$, homogeneous in $X_1,\dots,X_n$, such that $Q(z,f_1(z),\dots,f_n(z))\equiv 0$ and
\begin{equation} \label{theo_NS_PQrelation}
P(X_1,\dots,X_n)=Q(\alpha,X_1,\dots,X_n).
\end{equation}
\end{theorem}
Slightly later this theorem was proved with a completely different method by Andr\'e~\cite{Andre2000}. Using Andr\'e's method, Beukers~\cite{Beukers2006} has proved that in the statement of the Nesterenko-Shidlovsky Theorem above one can always take $S$ to be the set of zeros of $zT(z)$.

The class of $E$-functions is not the only example for which results similar to~\eqref{intro_trdegf_eq_trdegalpha} are sought. For example, of particular interest are sets of functions satisfying certain functional relations, such as relations of Mahler's type~\cite{Ni1986,Pellarin2009} or equations in $q$-differences~\cite{ATV2007,AV2003}. Another example is the so-called $G$-functions~\cite{Andre,Andre2000}, allied with $E$-functions. However, despite substantial progress, there are still many open problems outside the theory for $E$-functions.



The Nesterenko-Shidlovsky-Andr\'e-Beukers Theorem is best possible. For instance, if the functions~\eqref{intro_functions_f} have algebraic coefficients and admit an algebraic relation over $\mcc(z)$, that is if there is a polynomial $R\in\mcc(z)[X_1,\dots,X_n]$ vanishing at $\left(f_1(z),\dots,f_n(z)\right)$, one can verify 
that this is a relation over $\ol{\mqq}(z)$, i.e. the coefficients of $R$ are from $\ol{\mqq}(z)$. 
Multiplying $R$ by a common denominator of its coefficients we obtain 
a polynomial $R_1$ from $\ol{\mqq}[z][X_1,\dots,X_n]$. The polynomial $R_1(\alpha,X_1,\dots,X_n)$ vanishes at $\left(f_1(\alpha),\dots,f_n(\alpha)\right)$ and has algebraic coefficients, by construction. Thus polynomial relations~\eqref{theo_NS_PQrelation} from the theorem above are transmitted directly from functions to their values. So the best possible result regarding algebraic independence of the values $(f_1(\alpha),\dots,f_n(\alpha))$ one may hope to obtain is the result~\eqref{intro_trdegf_eq_trdegalpha}.


There exist powerful methods~\cite{PP1986,PP1997,PP_KF} which yield results of type~\eqref{intro_trdegf_eq_trdegalpha} for various sets of functions. It appears in practice that it is greatly preferable to measure the strength of algebraic independence between the functions~\eqref{intro_functions_f}. Loosely speaking, the reason is that the passage from algebraic independence of functions to the algebraic independence of their values usually involves a loss of information. It is very desirable to know how much one can lose.

The following question naturally arises: how one can measure the strength of algebraic independence? Let us start by considering the complex numbers
\begin{equation} \label{intro_numbers_x}
x_1,\dots,x_n\in\mcc^n.
\end{equation}
By definition, these complex numbers are algebraically independent if (and only if) for any non-zero polynomial $P\in\mqq[X_1,\dots,X_n]$ one has
\begin{equation*}
	P(x_1,\dots,x_n)\ne 0.
\end{equation*}
At first glance, one may try to use as  a measure of algebraic independence of numbers~\eqref{intro_numbers_x} the infimum of absolute value
\begin{equation} \label{intro_av_P}
|P(x_1,\dots,x_n)|
\end{equation}
over all allowed polynomials (that is $P\in\mqq[X_1,\dots,X_n]\setminus\{0\}$). However, it is easy to see that this infimum is always 0, at least if the degree and height\footnote{The notion of \emph{height} of a polynomial in fact has several meanings in the number theory. We understand it in a sense of Weil's logarithmic height. The reader can find the definition in subsection~\ref{sss_Heights}, for instance see~\eqref{defHsimple}.} of polynomials are unbounded. In view of this, it is natural to compare the rate of decreasing of the absolute value~\eqref{intro_av_P} with the corresponding values of \emph{degree} and \emph{height} of $P$.
\begin{definition}  \label{def_mia}
One says that a function $\phi:\mnn\times\mrr\rightarrow\mrr^+$ is a measure of algebraic independence of the set of complex numbers~\eqref{intro_numbers_x} if and only if the following inequality is verified for any non-zero polynomial $P\in\mzz[X_1,\dots,X_n]$
\begin{equation} \label{def_mai_lb}
	|P(x_1,\dots,x_n)|>\phi(\deg(P),h(P)).
\end{equation}
\end{definition}

In the case of formal power series, or analytic (at $z=0$) functions~\eqref{intro_functions_f} we can apply essentially the same definition. 
The natural choice of absolute value is that coming from the order of vanishing at the origin:
\begin{equation} \label{intro_expordz}
	|g(z)|=\exp(-\ordz g(z)).
\end{equation}
One readily verifies that this is indeed an (ultrametric) absolute value. The \emph{height} of polynomial $P\in\mcc[z][X_1,\dots,X_n]$ in this case is equal to $\deg_z P$.


So in the case of functional fields, the inequality~\eqref{def_mai_lb} takes the following form
\begin{equation} \label{def_mai_lb_functions}
	\exp\left(-\ordz P(f_1(z),\dots,f_n(z))\right)>\phi(\deg_z(P),\deg_{\ul{X}}(P)).
\end{equation}
Let
$$
F(\deg_z(P),\deg_{\ul{X}}(P)):=-\log\left(\phi(\deg_z(P),\deg_{\ul{X}}(P)\right).
$$
Then taking logarithms of both sides of~\eqref{def_mai_lb_functions} and changing signs in~\eqref{def_mai_lb_functions} we may rewrite this inequality as
\begin{equation*}
	\ordz P(f_1(z),\dots,f_n(z))<F(\deg_z(P),\deg_{\ul{X}}(P)),
\end{equation*}
which coincides with~\eqref{def_LM_a}. Hence the \emph{multiplicity estimate}~\eqref{def_LM_a} for a set of algebraically independent functions $f_1(z),\dots,f_n(z)$ is nothing else but the measure of algebraic independence of these functions. Note for instance that if we provide an upper bound $F(X,Y)$ in~\eqref{def_LM_a} with \emph{a slow rate of growth}, then we assure that the function $\phi$ in~\eqref{def_mai_lb_functions} has \emph{a slow rate of decrease}, so functions $(f_1(z),\dots,f_n(z))$ have a large measure of algebraic independence, in the sense we have explained just above.

We have a natural limit of results when seeking to improve the function $F$ in the r.h.s. of~\eqref{def_LM_a}. In any case, if functions $f_1,\dots,f_n$ are all algebraically independent we have
\begin{equation} \label{intro_F_lb}
F(X,Y)>\lceil(X+1)(Y+1)^n/n!\rceil.
\end{equation}
That is, for any $X,Y\in\mnn$\label{pl_F_is_large} we can construct a polynomial $P_{X,Y}$ of degree in $z$ at most $X$ and of degree in $X_1,\dots,X_n$ at most $Y$ verifying
\begin{align*}
\ordz P_{X,Y}(z,f_1(z),\dots,f_n(z))&\geq\lceil\frac{1}{n!}(X+1)(Y+1)^n\rceil\\&\geq\lceil\frac{1}{n!}\left(\deg_z(P)+1\right)\left(\deg_{\ul{X}}P+1\right)^n\rceil.
\end{align*}
To see this consider monomials $m_{k_0,k_1,\dots,k_n}=Z^{k_0}X_1^{k_1}\dots X_n^{k_n}$ with $0\leq k_0\leq X$ and $0\leq \sum_{i=1}^nk_i\leq Y$. There are $(X+1)\binom{Y+1+n}{n!}\geq\lceil(X+1)(Y+1)^n/n!\rceil$ of such monomials. The polynomial with indeterminate coefficients $c_{k_0,k_1,\dots,k_n}$
$$
Q(Z,X_1,\dots,X_n)=\sum_{\substack{0\leq k_0\leq X\\0\leq \sum_{i=1}^nk_i\leq Y}}c_{k_0,k_1,\dots,k_n}m_{k_0,k_1,\dots,k_n}
$$
has degree $\leq X$ in Z and $\leq Y$ in $X_1\dots,X_n$, in particular for any specialization of coefficients $c_{k_0,k_1,\dots,k_n}$. If we substitute $z,f_1(z),\dots,f_n(z)$ in $Q$ we obtain an analytic function
$$
g(z)=Q(z,f_1(z),\dots,f_n(z)).
$$
Every coefficient in Taylor series of this function is a linear form in $c_{k_0,k_1,\dots,k_n}$. Hence by basic linear algebra we can find a non-trivial set of coefficients $c_{k_0,k_1,\dots,k_n}$ such that the first $\lceil(X+1)(Y+1)^n/n!\rceil-1$ coefficients of $g(z)$ vanish. For this set of coefficients we have
$$
\ordz Q(z,f_1(z),\dots,f_n(z))=\ordz g(z)\geq\lceil(X+1)(Y+1)^n/n!\rceil,
$$
hence the claim.

If functions $f_1,\dots,f_n$ are algebraically dependent, we can not provide an upper bound~\eqref{def_LM_a} valid for all non-zero polynomials. We naturally have to exclude the ideal of polynomials vanishing at $\left(z,f_1(z),\dots,f_n(z)\right)$, we denote this ideal by $\idp_{\ul{f}}$. In this case, the considerations from linear algebra that justify~\eqref{intro_F_lb} can not be applied to the linear space of all the polynomials, we have to consider the linear space of polynomials of bi-degree bounded by $(X,Y)$ factorized by $\idp_{\ul{f}}$, the ideal of polynomials vanishing at $\left(z,f_1(z),\dots,f_n(z)\right)$. The dimension of this space is bounded from below~\cite{PP2000,PP} as a constant times $(X+1)(Y+1)^t$, where $t$ denotes the transcendence degree
\begin{equation} \label{intro_tr_degree}
t:=\trdeg_{\kk(z)}\kk(z)\left(f_1(z),\dots,f_n(z)\right).
\end{equation}
We naturally have to replace in~\eqref{intro_F_lb} the parameter $n$ by $t:=\trdeg_{\kk(z)}\kk(z)\left(f_1(z),\dots,f_n(z)\right)$.

To illustrate this at a more elementary level, let
\begin{equation} \label{intro_f_t}
f_1,\dots,f_t
\end{equation}
be algebraically independent (over $\kk(z)$) functions, and let $n>t$. Consider the $n$-tuple of functions
\begin{equation} \label{intro_f_nt}
(f_1,\dots,f_t,f_t,\dots,f_t),
\end{equation}
that is we complete the $t$-tuple~\eqref{intro_f_t} by $n-t$ copies of $f_t$. Clearly, the set of analytic functions that we can realize substituting the functions~\eqref{intro_f_nt} in the polynomials from $\kk[z][X_0,\dots,X_n]$ coincide with the set of functions that we can realize substituting the functions~\eqref{intro_f_t} in the polynomials from $\kk[z][X_0,\dots,X_t]$. In other terms, the additional copies of $f_t$ brings us no extra flexibility to increase the order of vanishing at $z=0$.

Thus the best possible function that we can have at the r.h.s. of~\eqref{def_LM_a} is
$$
C\left(\deg_{z}(P)+1\right)\left(\deg_{\ul{X}}(P)+1\right)^t,
$$
where $t$ is the transcendence degree~\eqref{intro_tr_degree}.


Another fact that we should keep in mind when proving the estimates of the type~\eqref{def_LM_a} is that there are sets of functions that refute any given r.h.s. $F(X,Y)$ in this inequality. It happens exactly when the field $\kk(z,f_1,\dots,f_n)$ contains (very) lacunary series. For  instance, let $g:\mrr^+\to\mrr^+$ be a function monotonically tending to infinity and satisfying $g(g(x))\geq g(x)+1$ for every $x\in\mrr^+$. Define $a_0=1$, $a_{n+1}=g(a_n)$ and $f_1(z)=\sum_{k=0}^{\infty}z^{a_k}$. Clearly the polynomial $P_N(z,X_1):=X_1-\sum_{k=0}^{N}z^{a_k}$ satisfies $\ordz P_N(z,f_1(z))>g(N)$, whilst $\deg P_N\leq N$.

At the same time, quite a lot of interest in multiplicity lemma comes from their potential applications to the problem of algebraic independence of values of analytic functions. Here it is worth mentioning that the result of the type~\eqref{intro_trdegf_eq_trdegalpha} does not hold, of course, for arbitrary sets of functions. Already in 1886 Weierstrass had constructed an example of a transcendental entire function $\mcc\rightarrow\mcc$ taking rational values at every rational point. In 1895 St\"ackel generalized this result, showing that for every countable subset $\Sigma\subset\mcc$ and every dense subset $D\subset\mcc$ there exists a transcendental entire function satisfying $f(\Sigma)\subset D$.

For all these reasons, when one aims to prove a multiplicity lemma or a result of the type~\eqref{intro_trdegf_eq_trdegalpha}, one is forced to introduce some extra assumptions on functions in question. Almost always these extra assumptions include the hypothesis of some functional relations satisfied by the set $f_1(z),\dots,f_n(z)$. In the modern theory of algebraic independence these functional relations most often take one of the following two types.

\begin{enumerate}

\item Differential system. Typically one consider a system
\begin{equation} \label{intro_diff_system}
\frac{d}{dz}\begin{pmatrix}f_1(z)\\\vdots\\ f_n(z)\end{pmatrix}=\begin{pmatrix}R_1(z,f_1(z),\dots,f_n(z))\\\vdots\\ R_n(z,f_1(z),\dots,f_n(z))\end{pmatrix},
\end{equation}
where $R_i(z,X_1,\dots,X_n)$ are rational functions (compare for instance with the hypothesis~\eqref{NS_theo_h_system} in the Nesterenko-Shidlovsky theorem)

\item Functional system. Typically it has a form
\begin{equation*}
\begin{pmatrix}f_1(p(z))\\\vdots\\ f_n(p(z))\end{pmatrix}=\begin{pmatrix}R_1(z,f_1(z),\dots,f_n(z))\\\vdots\\ R_n(z,f_1(z),\dots,f_n(z))\end{pmatrix},
\end{equation*}
where $p(z)$ is a rational function of the variable $z$ satisfying $p(0)=0$ and $R_i(z,X_1,\dots,X_n)$, $i=1,\dots,n$, are rational functions of the variables $z,X_1,\dots,X_n$.

For example, when $q$ is a complex number (satisfying $|q|>1$, say) and we set $p(z)=qz$, we find the general case of so called \emph{equations in $q$-differences}, currently widely studied~\cite{ATV2007,AV2003,Bertrand2007a,Bertrand2007,VZ2008}.

In the case $p(z)=z^d$, where $d\geq 2$ is an integer, we find a classical setup of Mahler's method. If we impose the weaker condition $\ordz p(z)\geq 2$ (with no extra assumption on the form of rational function $p(z)$), we find again Mahler's relations, this time understood in a broader sense~\cite{Ni1996,Pellarin2010,ThTopfer1995,EZ2010,EZ2011_2}.
\end{enumerate}

In all these cases there is a large variety of multiplicity lemmas established in various situations~\cite{Bertrand2007a,Bertrand2007,N1977,N1996,ThTopfer,EZ2010,EZ2011}.

The most general results link multiplicity lemmas with properties of ideals stable under an appropriate map.

For example, having a differential system~\eqref{intro_diff_system} we can define the differential operator $D:\kk[Z,X_1,\dots,X_n]\rightarrow\kk[Z,X_1,\dots,X_n]$ by
\begin{multline} \label{intro_def_D}
D(P)(Z,X_1,\dots,X_n)=A_0(Z,X_1,\dots,X_n)\frac{d}{dz}P(Z,X_1,\dots,X_n)\\+\sum_{i=1}^nA_i(Z,X_1,\dots,X_n)\frac{d}{dX_i}P(Z,X_1,\dots,X_n),
\end{multline}
where $A_i\in\mcc[Z,X_1,\dots,X_n]$ are polynomials such that the rational fractions $R_i$ in the system~\eqref{intro_diff_system} can be presented as $R_i=A_i/A_0$.
Note that the definition~\eqref{intro_def_D} assures
$$
D(P)(z,f_1(z),\dots,f_n(z))=A_0(z,f_1(z),\dots,f_n(z))\frac{d}{dz}P(z,f_1(z),\dots,f_n(z)).
$$
We say that an ideal $I$ of the ring $\mcc[Z,X_1,\dots,X_n]$ is $D$-stable iff $D(I)\subset I$.
The following theorem holds.
\begin{theorem}[Nesterenko, see Theorem~1.1 of Chapter~10, \cite{NP}] \label{theoNesterenko_classique}
Suppose that functions
\begin{equation*}
\ull{f} = (f_1(\b{z}),\dots,f_n(\b{z})) \in \mcc[[\b{z}]]^n
\end{equation*}
are analytic at the point $\b{z}=0$ and form a solution of the system~\eqref{intro_diff_system}.
If there exists a constant $K_0$ such that every $D$-stable prime ideal $\idp \subset \mcc[X_1',X_1,\dots,X_n]$,
$\idp\ne(0)$, satisfies
\begin{equation} \label{intro_ordIleqKdegI}
\min_{P \in \idp}\ordz P(\b{z},\ull{f}) \leq K_0,
\end{equation}
then there exists a constant $K_1>0$ such that for any polynomial $P \in
\mcc[X_1',X_1,\dots,X_n]$, $P\ne 0$, the following inequality holds
\begin{equation} \label{intro_LdMpolynome}
\ordz(P(\b{z},\ull{f})) \leq K_1(\deg_{\ul{X}'} P + 1)(\deg_{\ul{X}} P + 1)^n.
\end{equation}
\end{theorem}
\begin{remark}
Note that the upper bound~\eqref{intro_LdMpolynome} is the best possible, up to a multiplicative constant $K_1$
(see discussion on the page~\pageref{pl_F_is_large}).
\end{remark}
Note that the condition~\eqref{ordIleqKdegI} can be interpreted as the statement that all the differential ideals in the differential ring $(A,D)$ lies, in a certain sense, not too close to the functional point $(z,f_1(z),\dots,f_n(z))$. This statement was formalized by Nesterenko in~\cite{N1996}, he gave the name "$D$-property" to this phenomenon. In fact, this $D$-property is quite mysterious in nature: it seems hard to provide non-trivial examples of differential rings in characteristic 0 not satisfying it. At the same time, Nesterenko's theorem~\ref{theoNesterenko_classique} shows that this property ensures essentially optimal multiplicity lemmas, hence paving the way for the best possible results on algebraic independence. The fabulous example in this direction is the proof by Nesterenko of the fact that among four numbers
$$
\e^{2\pi iz}, E_2(z), E_4(z), E_6(z),
$$
where $z\in\mcc\setminus\{0\}$ verifies $0<|\e^{2\pi iz}|<1$ and $E_2$, $E_4$ and $E_6$ are Eisenstein series, at least three are algebraically independent over $\mqq$.

In the context of Mahler's method the corresponding general conditional result was conjectured but remained an open question~\cite{PP_KF} up to the recent time. In the same time, in the context of equations in $q$-differences D.Bertrand established its analogue, with a sharp control of multiplicative constant (corresponding to $K_1$ in~\eqref{intro_LdMpolynome}).

In our works~\cite{EZ2010,EZ2011} we established a common root for all such conditional results. We succeeded to introduce a natural formalism embedding all the situations mentioned above and to prove a conditional result analogous to Nesterenko's conditional multiplicity estimate cited above. In fact, being specialized to the case of differential systems our result gave the same conclusion as Nesterenko's theorem, and even more: in our result we replace the hypothesis~\eqref{intro_ordIleqKdegI} by a weaker one. Also, in the case of Mahler's method it gave the forecasted analogue of Nesterenko's theorem (again, in a reinforced form). Further analysis of stable ideals in polynomial ring allowed to deduce new multiplicity estimate within the context of Mahler's method and as a consequence to provide new results on algebraic independence~\cite{EZ2010,EZ2011_2}.

At the same time this general result has a drawback. It was established for the case of \emph{algebraically independent} functions $(f_1,\dots,f_n)$. However, in many situations of interest one may need a multiplicity estimate for \emph{algebraically dependent} functions. For example, in the context of Mahler's method, when applied to generating series of finite automata, it is quite usual to complete a set $f_1,\dots,f_r$ with some new functions, $f_{r+1},\dots,f_n$, in order to form a complete solution of a system of functional equations. These functions sometimes appear to be algebraically dependent with $f_1,\dots,f_r$ (over $\mcc(z)$). So even in the case when we aim to prove that the values $f_1(\alpha),\dots,f_r(\alpha)$ are all algebraically independent, it may appear to be very useful to be able to treat the case of algebraically dependent functions.

In this article we develop further and extend the techniques  elaborated in~\cite{EZ2010} and~\cite{EZ2011}. We obtain a general multiplicity estimate, see Theorem~\ref{LMGP}, optimal up to a multiplicative constant and applicable in the case of algebraically dependent functions.

There is a subtle point concerning the stable ideals in the case when the functions $f_1,\dots,f_n$ are algebraically dependent, or, using the notation $\idp_{\ul{f}}$ introduced above, if $\idp_{\ul{f}}\ne\{0\}$. The point is that in the case of differential system, as well as in our more general framework with the map $\phi$, the ideal $\idp_{\ul{f}}$ is $\phi$-stable. At the same time, the distance from the corresponding variety to $\ul{f}$ is 0, as $\idp_{\ul{f}}$ vanishes at $\ul{f}$. This fact immediately ensures that the $D$-property~\cite{N1996,NP}, as well as the weak $\phi$-property~\cite{EZ2010,EZ2011} do not hold. Hence the theorems stated in all the previous variants automatically have this important hypothesis failed, as far as functions in question are not algebraically independent.

However we neatened our formalism, allowing to exclude from the consideration all the ideals that vanish at $\ul{f}$, hence removing this difficulty.


Theorem~\ref{intro_LMGP} below presents a simplified version of the central result of this article (for the full statement, we refer the reader to Theorem~\ref{LMGP}). In this theorem we assume the following situation.
Let $\kk$ an algebraically closed field and let $\A=\kk[X_0',X_1',X_0,\dots,X_n]$ be a polynomial ring bi-graduated with respect to $\left(\deg_{\ul{X}'},\deg_{\ul{X}}\right)$. Consider a point
\begin{equation*}
\ul{f}=\left(1:z,1:f_1(z):\dots:f_n(z)\right)\in\mpp^1_{\kk[[z]]}\times\mpp^n_{\kk[[z]]}
\end{equation*}
and a map $\phi:\A\rightarrow\A$. We assume that the map $\phi$ is $\ul{f}$-admissible. This latter notion is introduced in Definition~\ref{def_admissible}. However on the first acquittance the reader may find more comfortable to postpone the reading of this definition and just keep in mind that both derivations and algebraic morphisms non-degenerated at the point $\ul{f}$ are $\ul{f}$-admissible, this notion is a common generalization for these two kinds of maps.

We denote by $t_{\ul{f}}$ the transcendence degree
\begin{equation} \label{def_r}
     t_{\ul{f}}:=\trdeg_{\kk(z)}\kk\left(f_1(z),\dots,f_n(z)\right),
\end{equation}
by $\idp_{\ul{f}}$ the bi-homogeneous ideal of polynomials from $\A$ vanishing at $\ul{f}$ and by $r_{\ul{f}}$ the rank of the ideal $\idp_{\ul{f}}$. Note that in view of these definitions we have $t_{\ul{f}}+r_{\ul{f}}=n+1$.

In the statement of Theorem~\ref{intro_LMGP} we use also the notation $m(I)$. It is introduced formally in Definition~\ref{definDePP}, informally it can be interpreted as a number of irreducible components (counted with multiplicities) in the variety $\V(I)$ associated to the ideal $I$. We also use the quantity $\ord_{\ull{f}}\idq$, introduced in Definition~\ref{defin_ord_xy}. Informally speaking, it measures how close is the point $\ul{f}$ to the zero locus of the ideal $\idq$: bigger is the quantity $\ord_{\ull{f}}\idq$, closer is the point $\ul{f}$ to the zero locus of the ideal $\idq$. At extremity, if all polynomials from $\idq$ vanish at $\ul{f}$, we have $\ord_{\ull{f}}\idq=+\infty$. On the first reading, the reader may find it comfortable to substitute $\min_{P\in\idq}P(\ul{f})$ instead of $\ord_{\ull{f}}\idq$. In many situations these two quantities coincide, and in any case we have $\min_{P\in\idq}P(\ul{f})\geq\ord_{\ull{f}}\idq$.

Finally, we say a few words on the $\left(\phi,\cK\right)$-property playing an important role in the statement of Theorem~\ref{intro_LMGP}. This property is described in Definition~\ref{def_weakDproperty}. The $\cK$ in the notation refer to a family of bi-homogeneous ideals of the ring $\A$. We say that the $\left(\phi,\cK\right)$-property holds, if for every ideal $I\subset\cK$ that verifies $\phi(I)\subset I$, we can find a prime factor $\idq\in\Ass\left(\A/I\right)$ that admits a nice upper bound for $\ord_{\ul{f}}\idq$ (informally speaking, $\idq$ is not too close to the point $\ul{f}$).
\begin{theorem}[Formal multiplicity lemma, simplified version]\label{intro_LMGP} Let $\kk$, $\A$, $\ul{f}$ and $\phi$ be as above, and let $C_0, C_1\in\mrr^+$. Assume that the map $\phi$ is $\ul{f}$-admissible. We denote by $\cK$ the set of all equidimensional bi-homogeneous ideals $I\subset\AnneauDePolynomes$ of rank $\geq 1+r_{\ul{f}}$, such that $\idp_{\ul{f}}\subsetneq I$, $\ul{f}\not\in\V(I)$ and $m(I)\leq C_m$ ($C_m$ is an absolute constant introduced in Definition~\ref{def_Cm}),
and moreover such that all its associated prime ideals satisfy
\begin{equation} \label{theoLMGP_condition_ordp_geq_C0}
\ord_{\ull{f}}\idq \geq C_0.
\end{equation}
Assume also that $\ul{f}$ has the $\left(\phi,\cK\right)$-property (see Definition~\ref{def_weakDproperty}).

Then there exists a constant $K>0$ such that for all $P \in
\AnneauDePolynomes$, satisfying $P(1,z,1,f_1(z),\dots,f_n(z))\ne 0$
satisfy also
\begin{equation} \label{LdMpolynome2}
\ordz(P(\ullt{f})) \leq K\left((\mu+\nu_0)(\deg_{\ul{X}'}P+1)+\nu_1\deg_{\ul{X}}P\right)\\ \times\mu^{n-1}(\deg_{\ul{X}} P + 1)^{t_{\ul{f}}}.
\end{equation}
\end{theorem}


Now it is a good point to say a few words about the ideas that we use in our proof. We present this short overview of our proof in a few of subsequent paragraphs. Note that at some points, for the sake of simplicity, we simplify some formulae, as compared to the formulae given in the main text. The reader will find later that the presented principles work as well if we use heavier variants from the main text.

We start with a polynomial
\begin{equation} \label{intro_explanation_start_point}
P(X_0',X_1',X_0,\dots,X_n)\in\A:=\kk[X_0',X_1',X_0,\dots,X_n]
\end{equation}
bi-homogeneous in groups of variables $\ul{X}'$ and $\ul{X}$. To establish a multiplicity lemma, we have to provide an upper bound for the order of vanishing of this polynomial at the functional point $\ul{f}:=(1,z,1,f_1(z),\dots,f_n(z))$. From some point of view, which we clarify in our article, the big order of vanishing $\ordz P(1,z,1,f_1(z),\dots,f_n(z))$ can be interpreted as a small projective distance from the point $\ul{f}\in\mpp^1\times\mpp^n$ to the (bi-projective) hypersurface defined by the polynomial $P$.

We use a transference lemma, recently established by P. Philippon~\cite{PP} (see section~\ref{section_transference_lemma} in this article), to find in this situation an algebraic point $\ul{\alpha}$ with a small projective distance to $\ul{f}$, lying in the zero locus of $P$ and such that every polynomial from $\A$ (see~\eqref{intro_explanation_start_point}) vanishing at $\ul{f}$ vanishes also at $\ul{\alpha}$. To this $\ul{\alpha}$, we associate a certain couple of integers $(\delta_0,\delta_1)$. We define this integers with two properties:
\begin{enumerate}
\item \label{intro_point_one} there exists a bi-homogeneous polynomial $Q\in\A$ of bi-degree $(\delta_0,\delta_1)$ vanishing at $\ul{\alpha}$ and does not vanishing at $\ul{f}$ and
\item the couple of integers $(\delta_0,\delta_1)$ minimizes (for the polynomials satisfying the point~\ref{intro_point_one}) a certain linear form related to $\idp$, the ideal of definition of $\ul{\alpha}$ (more precisely, it should minimize the linear form~\eqref{def_delta_condition_minimum_crossproduct} given below, where one substitutes $I:=\idp$ and the absolute positive constants $\mu$, $\nu_0$ and $\nu_1$ are defined with the general framework).
\end{enumerate}
To clarify the situation a little bit, we say that the analogue of the couple $(\delta_0,\delta_1)$ in the projective (and not bi-projective) space would be a minimal possible degree of a homogeneous polynomial vanishing at $\ul{\alpha}$ and not vanishing at $\ul{f}$.

In the subsequent paragraph, it is customary to use the notation $\idp_{\ul{f}}$ for the bi-homogeneous ideal of polynomials from $\A$ vanishing at $\ul{f}$.

It appears that the polynomials vanishing at $\alpha$, not vanishing at $\ul{f}$ and of a bi-degree comparable to $(\delta_0,\delta_1)$ have nice properties allowing us to complete our proof. In our general framework, we consider a map $\phi:\A\rightarrow\A$ such that the bi-degree of $\phi(P)$ can be controlled in terms of bi-degree of $P$, and the order of vanishing at $\ul{f}$ of $\phi(P)$ can be controlled in terms of order of vanishing at $\ul{f}$ of $P$. So, we introduce constants $\rho_i$, $i=1,\dots,n+1$ (see Definition~\ref{V_irho_i}), which depends on the transformation $\phi$ only. We consider ideals that we denote $I(V_i,\idp)$, $i=1,\dots,n+1$, generated as follows. We take an ideal generated by $\idp_{\ul{f}}$ and all the polynomials  of bi-degree at most $\rho_i$ times bigger than $(\delta_0,\delta_1)$ (see Definition~\ref{V_irho_i} for more precise formula) vanishing at $\ul{\alpha}$ and do not vanishing at $\ul{f}$, consider all its minimal primary factors that belong to the ideal $\idp$ (the ideal of definition of $\ul{\alpha}$) and take there intersection. From the geometrical point of view, we intersect the variety corresponding to $\idp_{\ul{f}}$ with all the (bi-projective) hypersurfaces defined over $\kk$, passing by $\ul{\alpha}$, do not passing by $\ul{f}$ and of bi-degree bounded by $(\rho_i\delta_0,\rho_i\delta_1)$, and in this complete intersection we choose the irreducible varieties passing by $\ul{\alpha}$.

The crucial property is that the the number of irreducible components, counted together with their multiplicities, in the variety corresponding to $I(V_i,\idp)$ is bounded by a constant that depends on $\rho_i$ only (see Lemma~\ref{LemmeCor14NumberW}). Using this property, we deduce that either the dimension of $I(V_i,\idp)$ is at most $n+1-i$ or at least one of the radical of primary components of this ideal is a $\phi$-stable ideal. In the latter case, we use the fact that all the primary components of $I(V_i,\idp)$ are contained in $\idp$, the ideal of definition of $\ul{\alpha}$. This property readily implies that all the components of $I(V_i,\idp)$ are sufficiently close to the point $\ul{f}$ (as the point $\ul{\alpha}$ was constructed to be close to $\ul{f}$ ). On the other hand, using a variant of B\'ezout's theorem we provide a nice control of bi-degree of $I(V_i,\idp)$, hence of all its primary components. These two bounds put together contradict our $(\phi,\cK)$-property, introduced in Definition~\ref{def_weakDproperty}. So, assuming in our main result, Theorem~\ref{LMGP}, that the $(\phi,\cK)$-property holds, we exclude this possibility.

To complete the proof, we remark that if the dimension of $I(V_i,\idp)$ is at most $n+1-i$ for $i=1,\dots,n+1$, then $I(V_{n+1},\idp)$ is necessarily 0-dimensional, and this is impossible as by construction all the minimal ideals of $I(V_{n+1},\idp)$ are contained in the one-dimensional ideal $\idp$.

To finish this introductory part, we remark that in fact we can consider, instead of one only map $\phi$, a (possibly infinite) family of maps $\phi_i$, $i\in I$. All we need to verify is that
\begin{enumerate}
\item all these maps satisfy the properties~\eqref{degphiQleqdegQ} and~\eqref{condition_T2_facile}, presented below, with uniform constants $\lambda$, $\mu$, $\nu_0$ and $\nu_1$ and
\item all these maps are locally correct at $\ul{f}$ (see Definition~\ref{def_locally_correct}); for example, this second point is automatically hold if these maps are derivations or dominant algebraic morphisms (see~\cite{EZ2011}, section~2.2).
\end{enumerate}
If these two conditions are satisfied, the proofs presented in this article can be transfered verbatim to the more general situation, with the transformation $\phi$ replaced by the family $\phi_i$, $i\in I$. In this case, the condition on $\phi$-stable ideals is replaced by the same condition on the ideals stable under all the $\phi_i$, $i\in I$. It seems that in certain situations it can restrain significantly the amount of ideals subject to be studied.

Nevertheless, in this paper we restrain our considerations to the case when we have only one transformation. The reason for this is, on the one hand, the purpose not to overcharge the paper with technical details, whilst it is already quite complicated from this point of view. On another hand, the reader who takes the effort to make out the proofs in this article will find it easy to pass to the case of many transformations.

\section{Framework, definitions and first properties}

\subsection{General framework} \label{subsection_general_framework}

As in~\cite{EZ2011}, we start the paper with the section recalling the general framework imposed in our studies (see~\cite{EZ2010}).

We denote by $\kk$ a (commutative) algebraically closed field of any characteristic, and by $\A$ a ring of polynomials with coefficients in $\kk$: $\A=\kk[X_0',X_1'][X_0,...,X_n]$. We consider the ring $\A$ as bi-graduated with respect to $\deg_{\ul{X}'}$ and $\deg_{\ul{X}}$.

\begin{remark}
The assumption that field $\kk$ is algebraically closed in fact is not a constraint. We readily extend our results to an arbitrary field using an embedding of a field in its algebraic closure, $\kk\subset\ol{\kk}$. We refer the reader to~\cite{EZ2011}, Remark~2.1 for more details.
\end{remark}


We fix a set of functions $f_1(z),\dots,f_n(z)\in\kk[[z]]$ and we denote by $\ul{f}$ the set $(1,z,1,f_1(z),\dots,f_n(z))$. Note that one can consider $\ul{f}$ as a system of projective coordinates of a point $(1:\b{z},1:f_1(\b{z}):...:f_n(\b{z})) \in \mpp^1_{\kk[[\b{z}]]}\times\mpp^n_{\kk[[\b{z}]]}$. By a slight abuse of notation we also sometimes denote $\ul{f}=(1:\b{z},1:f_1(\b{z}):...:f_n(\b{z})) \in \mpp^1_{\kk[[\b{z}]]}\times\mpp^n_{\kk[[\b{z}]]}$. Our final aim in this article is to provide~\emph{multiplicity estimate} for functions $f_1(z),\dots,f_n(z)$.

The main difference with our previous article~\cite{EZ2011} consists in the fact that we drop the assumption that all the functions $f_1,\dots f_n$ are algebraically independent over $\kk(z)$. The following definition introduces the notions that allow to control this dependence.
\begin{definition} \label{def_tf}
 We denote by $\idpf$ the bi-homogeneous ideal of polynomials $P\in\A$ satisfying $P(\ul{f})=0$. Also, we denote by $C_{\ul{f}}\geq 1$ a constant such that $C_{\ul{f}}\geq 1$ and such that the ideal $\idpf$ is defined by polynomials of bi-degree bounded by $C_{\ul{f}}$, $(\deg_{\ul{X}'}P,\deg_{\ul{X}}P)\leq\left(C_{\ul{f}},C_{\ul{f}}\right)$. We define by $t=t_{\ul{f}}$ the transcendence degree
\begin{equation*}
t=t_{\ul{f}}:=\trdeg_{\mcc(z)}\mcc\left(f_1(z),\dots,f_n(z)\right)
\end{equation*}
and by $r_{\ul{f}}$ the rank of the bi-homogeneous ideal $\idp_{\ul{f}}$.
In view of these definitions we have the equality
$$
t_{\ul{f}}+r_{\ul{f}}=n.
$$
\end{definition}

To provide multiplicity estimate for $\ul{f}$ we need an additional structure. This structure will be encoded in properties of a map $\phi$ below. We do not suppose \emph{a priori} that $\phi$ respects any classical structure defined on $\A$, for example that one of the polynomial ring. Instead we impose some conditions on this map that are suitable for our purposes and applications we have in mind. For instance, we assume~(\ref{degphiQleqdegQ}) and~(\ref{condition_T2_facile}) below.

We fix a bi-homogeneous map $\phi:\A\rightarrow\A$ such that for all bi-homogeneous polynomial
$Q\in\AnneauDePolynomes$ one has
\begin{equation} \label{degphiQleqdegQ}
\begin{aligned}
&\deg_{\ul{X}} \phi(Q) \leq \mu \deg_{\ul{X}} Q,\\
&\deg_{\ul{X}'} \phi(Q) \leq \nu_0 \deg_{\ul{X}'} Q + \nu_1 \deg_{\ul{X}} Q
\end{aligned}
\end{equation}
with some positive constants $\mu, \nu_0>0$ and a non-negative constant $\nu_1$.

\begin{notation}
We denote by $\phi^N$ the $N$-th iteration of the map $\phi$.
\end{notation}

Using recurrence on the hypothesis~\eqref{degphiQleqdegQ} one readily establishes the following lemma.
\begin{lemma} \label{majorationphinQ} Let $N$ be a positive integer and $Q\in\AnneauDePolynomes$ be a bi-homogeneous polynomial. Then,
\begin{eqnarray}
  \deg_{\ul{X}}\phi^N(Q) &\leq& \mu^N\deg_{\ul{X}}Q, \label{majorationdegXphinQ} \\
  \deg_{\ul{X'}}\phi^N(Q) &\leq& \nu_0^N\deg_{\ul{X'}}Q+\nu_1\left(\sum_{i=0}^{N-1}\nu_0^{N-i-1}\mu^i\right)\deg_{\ul{X}}Q. \label{majorationdegXprimephinQ}
\end{eqnarray}
\end{lemma}
\begin{proof}
See~\cite{EZ2011}, Lemma~2.5.
\end{proof}

\hidden{
\begin{notabene}
Constants $\mu$, $\nu_0$ and $\nu_1$ appear quite often in this text, for example they are used in Definition~\ref{def_delta} further within this section. All the time these letters denote \emph{the same} constants. That is, we apply all the machinery presented in this paper every time to \emph{just one} map $\phi$ (though this map is arbitrary in the limits described in this section), in the sense that definitions such as Definition~\ref{def_delta} depend on the choice of $\phi$, more precisely they depend on the constants $\mu$, $\nu_0$ and $\nu_1$ in~\eqref{degphiQleqdegQ}. This fact does not restrain the generality of our considerations, but it may be useful to note at this point that constants $\mu$, $\nu_0$ and $\nu_1$ from property~\eqref{degphiQleqdegQ} are considered as absolute in what follows.
\end{notabene}
}

We assume that there exist two constants $\lambda>0$ and $K_{\lambda}\geq 0$ such that
\begin{equation} \label{condition_T2_facile}
    \ordz \phi(Q)(\ul{f}) \geq \lambda \, \ordz(Q(\ul{f})).
\end{equation}
for all bi-homogeneous polynomials $Q\in\AnneauDePolynomes$ satisfying $\ordz(Q(\ul{f}))\geq K_{\lambda}$.

Two typical examples of a map $\phi$ satisfying~\eqref{degphiQleqdegQ} and~\eqref{condition_T2_facile} are derivations and algebraic morphisms.

Our principal result, Theorem~\ref{LMGP}, is proved for maps satisfying these assumptions, as well as one additional assumption described in Definition~\ref{defin_phiestcorrecte}.

The proof can be found in~\cite{EZ2011} (see Lemma~2.5 \emph{loc.cit.}).

\begin{remark}
We will need to consider $\phi$ as acting on $\kk[\b{z}][X_0:...:X_n]$ by setting
\begin{equation}
\b{\phi}(\b{Q})=\phi\left(X_0'^{\deg_{\b{z}}\b{Q}}\b{Q}(X_1'/X_0')(X_0:...:X_n)\right)\Bigg|_{(X_0':X_1')=(1,\b{z})}
\end{equation}
for all $Q\in\kk[\b{z}][X_0:...:X_n]$ homogeneous in~$X_0,\dots,X_n$.

This map $\b{\phi}$ satisfies
\begin{equation} \label{hphiQleqhQ}
\begin{aligned}
\deg_{\ul{X}} \b{\phi}(\b{Q}) &\leq \mu \deg_{\ul{X}} \b{Q},\\
h(\b{\phi}(\b{Q}))=\deg_{\b{z}} \b{\phi}(\b{Q}) &\leq \nu_0 \deg_{\b{z}} \b{Q} + \nu_1\deg_{\ul{X}} \b{Q} \\&\leq \nu_0 h(\b{Q}) + \nu_1\deg\b{Q}.
\end{aligned}
\end{equation}
\end{remark}

\subsection{Definitions and properties related to commutative algebra} \label{definitions_comm_algebra}

\begin{definition}
Let $I\subset\AnneauDePolynomes$ be a bi-homogeneous ideal.  We denote by $\V(I)$ the sub-scheme of $\mpp^1\times\mpp^n$ defined by $I$.
Conversely, for any sub-scheme $V$ of $\mpp^1\times\mpp^n$ we denote $\I(V)$ the bi-homogeneous saturated ideal in $\AnneauDePolynomes$ that defines $V$.
\end{definition}

\begin{definition} \label{def_I}
Let $V$ be a $\kk$-linear subspace of $\A$ and $\idp\subset\A$ a prime ideal. We define $I(V,\idp)$ to be the smallest bi-homogeneous ideal of $\AnneauDePolynomes$ containing $(V\AnneauDePolynomes_{\idp})\cap\AnneauDePolynomes$,
where $\AnneauDePolynomes_{\idp}$ denotes the localization of $\AnneauDePolynomes$ by $\idp$ and $V\AnneauDePolynomes_{\idp}$ denotes
the ideal generated in $\AnneauDePolynomes_{\idp}$ by elements of $V$.
\end{definition}

\begin{remark}
    Ideal $I(V,\idp)$ is the intersection of the primary components of $V\A$ contained in $\idp$.
\end{remark}
\begin{definition}\label{definIdealTstable}
We say that an ideal $I \subset \AnneauDePolynomes$ is \emph{$\phi$-stable} if and only if $\phi(I) \subset I$.
\end{definition}
\begin{definition} \label{def_eqI} Let $I$ be a bi-homogeneous ideal of the ring $\A$ and
\begin{equation}
I = \mathcal{Q}_1 \cap \dots \cap \mathcal{Q}_r \cap
\mathcal{Q}_{r+1} \cap ... \cap \mathcal{Q}_{s}
\end{equation}
be its primary decomposition, where
$\mathcal{Q}_1$,...,$\mathcal{Q}_r$ are the bi-homogeneous primary ideals associated to the ideals of minimal rank (i.e. of rank
$\rg(I)$) and $\mathcal{Q}_{r+1}$,...,$\mathcal{Q}_{s}$
correspond to the components of rank
strictly bigger than $\rg(I)$.

We denote by
\begin{equation}
 \eq(I) \eqdef \mathcal{Q}_1 \cap \dots \cap \mathcal{Q}_r
\end{equation}
the equidimensional part of the minimal rank of $I$.
\end{definition}



We give now a preliminary definition, it will be needed in Definition~\ref{def_admissible}, which introduces a property important for our main result.

\begin{definition} \label{defin_phiestcorrecte} We say that a map $\phi:\A\rightarrow\A$ is {\it correct with respect to the ideal $\idp\subset\AnneauDePolynomes$} if for every ideal
$I$, such that all its associated primes are contained in $\idp$, the inclusion
\begin{equation}\label{defin_phi_phiI_subset_eqI}
    \phi(I)\subset\eq(I)
\end{equation}
implies
\begin{equation}\label{defin_phi_phieqI_subset_eqI}
    \phi(\eq(I))\subset\eq(I)
\end{equation}
(recall that $\eq(I)$ is introduced in Definition~\ref{def_eqI}).
\end{definition}


Two important examples of correct morphisms are derivations and (dominant) algebraic morphisms (see~\cite{EZ2011}, section~2.2 for proofs and some more discussions on the class of correct maps).

\begin{definition}\label{definDePP}\begin{enumerate}
  \item Let $\idp$ be a prime ideal of the ring $\AnneauDePolynomes$, $V$ a $\kk$-linear subspace
of $\AnneauDePolynomes$ and $\phi$ a (set-theoretical) map of $\AnneauDePolynomes$ to itself. Then
\begin{equation} \label{definDePP_defin_e}
 e_{\phi}(V,\idp) \eqdef \max(e\,\vline\,\rg\left((V+\phi(V)+...+{\phi}^e(V))\AnneauDePolynomes_{\idp}\right)=\rg\left(V\AnneauDePolynomes_{\idp}\right)).
\end{equation}
\item Let $\mathcal{R}$ be a ring and $M$ be an $\mathcal{R}$-module. We denote by $l_{\mathcal{R}}(M)$ the length of $M$ (see p.~72 of~\cite{Eis} for the definition). In fact we shall use this definition only in the case $\mathcal{R}=\AnneauDePolynomes_{\idp}$ and $M=(\AnneauDePolynomes/I)_{\idp}$, where $I$ denotes an ideal of $\A$.
\item Let $I$ be a proper ideal of the ring $\AnneauDePolynomes$,
   \begin{equation} \label{definDePP_defin_m}
    m(I)=m(\eq(I)) \eqdef \sum_{\idp\in\Spec(\AnneauDePolynomes)\,\vline\,\rg(\idp)=\rg(I)}l_{\AnneauDePolynomes_{\idp}}((\AnneauDePolynomes/I)_{\idp}) \in \mnn^*.
   \end{equation}
\end{enumerate}
\end{definition}

Note that the quantity $m(I)$ is the number of primary components of $I$ counted with their length as a multiplicity.

\subsection{Definitions and properties related to multi-projective diophantine geometry} \label{definitions_multiprojective_dg}

In this section, we shall see the notions of \emph{(bi-)degree} and \emph{height} of a variety. We give here several properties of these quantities that we shall use later. For a more detailed introduction the reader is invited to consult Chapters~5 and~7 of~\cite{NP} or Chapter~1 of~\cite{EZ2010}.
\subsubsection{Heights} \label{sss_Heights}
Let $K$ be an infinite field. Assume that there exists a family $\MM_{K}$ of absolute values, $\left(|\cdot|_v\right)_{v\in \MM_{K}}$, satisfying the \emph{product formula} with exponents $n_v$:
$$
\prod_{v\in \MM_K}|\alpha|_v^{n_v}=1\text{ for every }\alpha\in K\setminus\{0\}.
$$
It is a classical result that if $L$ is a finite extension of $K$, then absolute values from $\MM_K$ can be extended to a form a family $M_L$ of absolute values on $L$, satisfying a product formula with exponents $(n_w)_{w\in \MM_L}$:
$$
\prod_{w\in \MM_L}|\alpha|_w^{n_w}=1\text{ for every }\alpha\in L\setminus\{0\}.
$$
In this situation we can define a notion of the \emph{height} for different objects defined over $\ol{K}$.


\begin{example}
\begin{enumerate}
           \item $K=\mqq$, $\MM_K=\{\text{prime numbers}\}\cup\{\infty\}$. If $p$ is a prime, $|\cdot|_{p}$ is equal to the $p$-adic absolute value (normalized to have $|p|_p=p^{-1}$) and $|\cdot|_{\infty}$ is equal to the usual archimedian absolute value over $\mqq$. The product formula in this case is a consequence of the fundamental  theorem of arithmetics.
           \item $K=\kk(\b{z})$, where $\kk$ denotes an algebraically close (commutative) field, and $\MM_K=\mpp^1_{\kk}$. We associate to every element $v\in\MM_K$ the absolute value $\exp(-\ord_{v})$. The product formula results in this case from the uniqueness of the polynomial factorization.
         \end{enumerate}
\end{example}

To every \emph{ultrametric} place $v$ we naturally associate the
valuation $\ord_v$. For the commodity of notation in what follows, we introduce the following notation for  \emph{every} valuation, archimedean as well as ultrametric ones. We define for every absolute value $|\cdot|_v$ (ultrametric
or archimedian) $\ord_v\alpha:=-\log |\alpha|_v$ for every $\alpha \in K^*$.


\paragraph{Height of elements.} We start by recalling the notion of the \emph{height} of an element from $\ol{K}$. So let's take a finite extension $L\supset K$ and let $\MM_L$ denotes a family of places of $L$ extending $\MM_K$. For every $v\in\MM_L$ we denote by $L_v$ the completion of $L$ with respect to $v$ and
$$
n_v:=[L_v:K_v].
$$
For every $\b{\alpha} \in L$ we define
\begin{equation} \label{defHsimple}
 h_L(\b{\alpha}) := -\frac{1}{[L:K]}\sum_{v\in\MM_L}n_v \min(0,\ord_v(\b{\alpha})),
\end{equation}
In fact this definition does not depend on the extension $L$ chosen in the beginning: if $L \subset L' \subset \overline{K}$, $[L':L] < +\infty$ and $\b{\alpha} \in L$ we have $h_{L'}(\b{\alpha})=h_{L}(\b{\alpha})$. 

More generally, let ${\alpha} \in \mpp^n_{\ol{K}}$, we fix a representative $\ul{\alpha}\in\ol{K}^{n+1}$ of ${\alpha}$ and a finite extension $L$ of $K$ such that all the coordinates of $\ull{\alpha}$ belong to this extension. We set
\begin{equation} \label{defHpoint}
 h({\alpha}) \eqdef -\frac{1}{[L:K]}\sum_{v\in\MM_L}n_v \min(\ord_v(\b{\alpha}_0),...,\ord_v(\b{\alpha}_n)).
\end{equation}
One readily verifies that this definition does not depend on the choice of the representative $\ull{\alpha}$ neither on the choice of the extension $L$. 

In view of these definitions, for any $\alpha\in L$ we have $h_L(\alpha)=h(1:\alpha)$.

\paragraph{Height of forms.} Let $L$ be a finite extension of $K$ and let ${F}\in L[\ul{u}^{(1)},\dots,\ul{u}^{(n)}]$ be a non-zero multihomogeneous form. 
For every place $v$ of $L$ (archimedean or non-archimedian) we denote by $M_v({F})$ the maximum of the $v$-adic norm 
of the coefficients of ${F}$. 

\hidden{===================================
On the other hand, for every archimedian place $v$, we have by the theorem of Gelfand-Tornheim (see~\cite{A1956}, p. 45 and 67) an embedding $\sigma_v: K_v\rightarrow\mcc$ such that $|\alpha|_v=|\sigma_v(\alpha)|$, where $|\cdot|$ denotes the usual (archimedean) absolute value on $\mcc$. 
In this case we define $M_v(F):=M(\sigma_v(F))$, where $M(\cdot)$ denotes the measure recalled in~\cite{NP}, chap.~7, p.~97:
\begin{multline*}
    \log M_v(\sigma_v(P))=\int_{S_{n_1+1}(1)\times\dots\times S_{n_r+1}(1)}\log|\sigma_v(F)|\sigma_{n_1+1}\wedge\dots\wedge\sigma_{n_r+1}\\+\sum_{i=1}^r\deg_{\ul{u}^{(i)}}F\sum_{j=1}^{n_i}\frac{1}{2j}
\end{multline*}
where $S_{n+1}(1)$ denotes the sphere of the radius $1$ in $\mcc^{n+1}$ and $\sigma_{n+1}$ is the normalized Haar measure of the total mass 1 on $S_{n+1}(1)$.
=========================================}

We define then the \emph{height of the form ${F}$} to be
\begin{equation} \label{def_hF}
    h({F})\eqdef\frac{1}{[L:K]}\sum_{v\in\MM_L}n_v\log M_v({F}).
\end{equation}
We complete this definition by $h(0)=0$. Note that replacing $L$ by its finite extension does not affect the value $h(\b{F})$.


\subsubsection{Bi-degrees}
\begin{enumerate}
  \item \label{def_bd_p_hypersurface} In the case of an hypersurface, i.e. if the variery $V\subset\mpp^1_{\kk}\times\mpp^n_{\kk}$ is the locus of the zeros of a bi-homogeneous polynomial $P\in\A$, the bi-degree is a couple of integers $\left(\deg_{\ul{X}'}P,\deg_{\ul{X}}P\right)$. In this case it is common to write also $\deg_{1,n-1}V:=\deg_{\ul{X}}P$  and $\deg_{0,n}V:=\deg_{\ul{X}'}P$. In general the bi-degree of a variety $V\subset\mpp^1\times\mpp^n$ is a couple of integers denoted often as $\left(\deg_{0,\dim(V)}V,\deg_{1,\dim(V)-1}V\right)$. This notation is explained in Chapter~5 of~\cite{NP}.
  \item If $V=V_1\cup\dots\cup V_r$ is a decomposition of $V$ in a union of irreducible components, we have
  \begin{equation*}
    \dim_{i,n-i}V=\sum_{j=1}^r\dim_{i,n-i}V_j,\quad i=0,1.
  \end{equation*}
  \item For any irreducible variety $V\subset\mpp^1\times\mpp^n$ and any hypersurface $Z\subset\mpp^1\times\mpp^n$ of bi-degree $(a,b)$, such that $V$ and $Z$ intersect properly, there exists a variety $W$ such that its zero locus coincides with intersection of zero loci of $V$ and $Z$ (hence $\dim W=\dim V-1$), and $W$ satisfies
      \begin{eqnarray}
        \dd_{(1,\dim(V)-2)}(W) &=& b\cdot\dd_{(1,\dim(V)-1)}(V), \label{BT_dll} \\
        \dd_{(0,\dim(V)-1)}(W) &=& a\cdot\dd_{(1,\dim(V)-1)}(V) + b\cdot\dd_{(0,\dim(V))}(V). \label{BT_doo}
      \end{eqnarray}
      We shall denote such a variety $W$ as $V\cap Z$.
  \item Let $W\subset\mpp^n_{\kk(z)}$ be a subvariety. We can replace $(1:z)$ by $(X_0':X_1')$ transforming $W$ into a subvariety $\tilde{W}\subset\mpp^1_{\kk}\times\mpp^n_{\kk}$. Our point here is that $\tilde{W}$ is a bi-projective variety over $\kk$, whilst $W$ is a projective variety over $\kk(z)$. In this case we have a direct link between the \emph{height} and the \emph{degree} of $W$ on one side and the bi-degree of $\tilde{W}$ on another side. Notably, the \emph{height} of $W$ equals $h(W)=\dd_{(0,\dim(\tilde{W}))}(\tilde{W})$ and the \emph{degree} of $W$ is $\deg(W)=\dd_{(1,\dim(\tilde{W})-1)}(\tilde{W})$.
  \item We can associate to any bi-homogeneous ideal $I\subset\A$ (resp. any homogeneous ideal $J\subset\kk[z][X_0,\dots,X_n]$) a bi-projective (resp. projective) variety \V(I), thus defining $\deg_{i,n+1-\rk(I)-i}I$, $i=0,1$ (resp. $\deg(I)$ and $h(I)$).
\end{enumerate}


We shall use later the following lemma. It is a variant of the so-called \emph{B\'ezout's theorem}.
\begin{lemma} \label{lemma_BT}
Suppose that a bi-homogeneous ideal $I\subset\A$ has rank $r$, contains a bi-homogeneous ideal $\idp\subset I$ of rank $r_{\idp}$ and is generated by $\idp$ and $r-r_{\idp}$ bi-homogeneous polynomials of bi-degree upper bounded by $(a,b)$. Then
\begin{eqnarray}
	\dd_{(1,n-r)}(I)&\leq& \deg_{(1,n-r_{\idp})}\idp\cdot b^{r-r_{\idp}} \label{BT_dll_v2},\\
	\dd_{(0,n-r+1)}(I) &\leq&(r-r_{\idp})\deg_{(1,n-r_{\idp})}\idp\cdot a \cdot b^{r-r_{\idp}-1}\\\nonumber&&+\deg_{(0,n-r_{\idp}+1)}\idp\cdot b^{r-r_{\idp}}. \label{BT_doo_v2}
\end{eqnarray}
\end{lemma}
\begin{proof}
This is a consequence of Propositions~3.4 and~3.6 of Chapter~5, \cite{NP}.
\end{proof}

We shall regularly use the valuation $\ordz$ on the ring $\kk[[z]]$ of formal power series. This valuation induces the notions $\Ord({x},V)$ and $\ord({x},V)$, both measuring how far a point $x$ is in a (multi-)projective space from a variety $V$ belonging to the same space. In some related articles, these quantities may be denoted by $\Dist({x},V):=\exp(-\Ord({x},V))$ and $\dist({x},V):=\exp(-\ord({x},V))$. Precise definitions could be found in~\cite{NP}, chapter~7, \S~4 and~\cite{EZ2010}, chapter~1, \S~3. We shall interchangeably use the notation $\Ord_{x}V:=\Ord({x},V)$ and $\ord_{x}V:=\ord({x},V)$

In order to make this article self-contained we introduce briefly these notions. In this article, we define the quantity $\Ord$ only in cases when $V$ is either 0-dimensional or a hypersurface. This is the only cases when we make use of $\Ord$. We refer the reader to~\cite{NP}, Chapter~7, \S~4 and~\cite{EZ2010}, Chapter~1, \S~3 for the general treatment.

\begin{definition} \label{defin_ord_xy} \begin{enumerate}
  \item If $\ul{x}=(x_0,\dots,x_n)\in\kk((z))^{n+1}$, we define $\ordz\ul{x}=\min_{i=0,\dots,n}\ordz x_i$.
  \item Let $x,y\in\mpp_{\kk((z))}^n$ be two points and $\ul{x}$ and \ul{y} be systems of projective coordinates respectively for $x$ and $y$. We define $\ul{x}\wedge\ul{y}$ to be a vector with $n(n-1)/2$ coordinates $\left(x_iy_j-x_jy_i\right)_{1\leq i<j\leq n}$ (the ordering of coordinates $x_iy_j-x_jy_i$ of this vector is not important for our purposes). Finally, we define
      \begin{equation} \label{def_ord_xy}
        \ordz(x,y):=\ordz(\ul{x}\wedge\ul{y})-\ordz \ul{x} - \ordz\ul{y}.
      \end{equation}
      One readily verifies that the r.h.s. in~(\ref{def_ord_xy}) does not depend on the choice of systems of projective coordinates for $x$ and $y$.
  \item Let $x,y\in\mpp^1_{\kk((z))}\times\mpp^n_{\kk((z))}$ and $\pi_1$ (resp. $\pi_n$) be a canonical projection of $\mpp^1_{\kk((z))}\times\mpp^n_{\kk((z))}$ to $\mpp^1_{\kk((z))}$ (resp. $\mpp^n_{\kk((z))}$). We define
       \begin{equation} \label{def_ord_xy_biprojectif}
        \ordz(x,y):=\min_{i=1,n}\ordz\left(\pi_i(x),\pi_i(y)\right).
      \end{equation}
  \item Let $V\subset\mpp^1\times\mpp^n$ (or $V\subset\mpp^n$) be a variety. We define
      \begin{equation} \label{def_ord_Vx}
        \ordz(x,V):=\max_{y\in V}\ordz\left(x,y\right).
      \end{equation}
  \item Sometimes we shall write simply $\ord(x,y)$, $\ord(x,V)$, $\Ord(x,V)$ etc. instead of $\ordz(x,y)$, $\ordz(x,V)$, $\Ord_{z=0}(x,V)$ etc. This will not create any ambiguity because we shall be interested in only one valuation $\ordz$, so all the derived constructions, such as $\ord(x,y)$ and $\ord(x,V)$, will refer always to this valuation.
\item We shall use the notation $\ord_x(V)$ to refer to $\ord(x,V)=\ordz(x,V)$ introduced in this definition.
\end{enumerate}
\end{definition}
We proceed to introduce $\Ord(x,V)$ for the cases when $\dim(V)=0$ or $V$ is a hypersurface. 
\begin{enumerate}
  \item \label{defin_ordOrd_Ord1} If $V$ is 0-dimensional over $\ol{\kk((z))}$, it can be represented as a union of $r$ points $y_1,\dots,y_r$ (in fact, $r=\deg(V)$) and we define $\Ord(x,V):=\sum_{i=1}^r\Ord(x,y_i)$. In particular, if $V$ contains just one point over $\ol{\kk((z))}$, we set $\Ord(x,y)=\ord(x,y)$.
\hidden{
  \item \label{defin_ordOrd_Ord2} Let $V$ be a variety of dimension $d\geq 0$. We consider an intersection of $V$ with $d$ hyperplanes $U_1,\dots,U_d$ in general position, it is a 0-dimensional variety $V\cap U_1\cap\dots\cap U_d$. In particular, the quantity $\Ord(x,V\cap U_1\cap\dots\cap U_d)$ is well-defined as a function of $(U_1,\dots,U_d)$. We denote $\Ord(x,V)$ the supremum of this function for all the sets of hyperplanes $U_1,\dots,U_d$ in general position.
  \item \label{defin_ordOrd_Ord3} More generally, we can replace in point~(\ref{defin_ordOrd_Ord2}) a set of $d$ hyperplanes in general position by a set of $d$ hypersurfaces, of degrees $(\delta_1,\dots,\delta_d)\in\mnn^n$, in general position. In this case we obtain a notion usually denoted as $\Ord_{(\delta_1,\dots,\delta_d)}(x,V)$.
}
  \item \label{defin_ordOrd_Ord4} If $V=\Z(F)$, where $F\in\kk[X_0',X_1'][X_0,\dots,X_n]$, then for any system of bi-projective coordinates $\ul{x}=(x_0',x_1',x_0,\dots,x_n)$ of $x$ we have
  \begin{equation*}
    \Ord(x,V)=\ordz F(\ul{x})-\left(\ordz\ul{x}\right)^{\deg F}
  \end{equation*}
  (see p.~89 of~\cite{NP}).
\end{enumerate}


Now we are ready to introduce the key notions of $\ul{f}$-admissible map and of \emph{$(\phi,\cK)$-property} (see Definitions~\ref{def_admissible} and~\ref{def_weakDproperty} below). These notions are used in the statement of our main result. We start with a preliminary definition.

\begin{definition} \label{def_locally_correct}
Let $\kk$ be a field and $\ul{f}=(1,z,1,f_1,\dots,f_n)\in\kk[[z]]^{n+3}$. Let $\phi:\A\rightarrow\A$ be a bi-homogeneous self-map of a polynomial ring $\A=\kk[X_0',X_1'][X_0,\dots,X_n]$ 
and $C_0 \in \mrr^+$ be a constant
such that for all bi-homogeneous prime ideal $\idq\subset\AnneauDePolynomes$ of rank $n$ one has
\begin{equation} \label{intro_theoLMGP_condition_de_correctitude}
\ord_{\ull{f}}\idq\geq C_0 \Rightarrow \mbox{ the map $\phi$ is correct with respect to }\idq.
\end{equation}
In this situation we say that $\phi$ is \emph{locally correct at} $\ull{f}$.
\end{definition}

\begin{definition} \label{def_admissible}
We say that a bi-homogeneous map $\phi:\A\rightarrow\A$ is \emph{$\ul{f}$-admissible} (or simply \emph{admissible}) if it is locally correct at $\ul{f}$ and satisfies~(\ref{degphiQleqdegQ}) and~(\ref{condition_T2_facile}).
\end{definition}
\begin{remark} \label{remark_def_admissible}
Corollary~2.24 of~\cite{EZ2011} implies that derivations are $\ul{f}$-admissible maps for arbitrary $\ul{f}$. Also, corollary~2.25 of~\cite{EZ2011} implies that under some mild restrictions (essentially, to be non-degenerate in the neighbourhood of the point $\ul{f}$) algebraic morphism $\T^*$ is an $\ul{f}$-admissible map.
\end{remark}

\begin{definition} \label{def_weakDproperty}
Let $\A$ be a polynomial ring and $\phi:\A\rightarrow\A$ a map. 
Let $\cK$ be a subset of the set of ideals of $\A$.

Suppose that there exists a
constant $K_0 \in \mrr^{+}$ (depending on $\cK$, $\phi$ and $\ull{f}$ only) with the following property: for every 
ideal $I\in\cK$ that is $\phi$-stable (i.e. $\phi(I)\subset I$) 
there exists a prime factor $\idq\in\Ass(\AnneauDePolynomes/I)$ satisfying
\begin{equation} \label{def_RelMinN2} 
\ord_{\ull{f}}(\idq) < K_0\left(\dd_{(0, n-\rg\idq+1)}(\idq)+\dd_{(1, n-\rg\idq)}(\idq)\right).
\end{equation}
In this situation we say that \emph{$\ul{f}$ has the $\left(\phi,\cK\right)$-property}, or, if the choice of $\ul{f}$ is obvious, we say also that \emph{one has $\left(\phi,\cK\right)$-property} (or also that couple $\left(\phi,\cK\right)$ satisfies the \emph{weak $\phi$-property}).
\end{definition}
\begin{remark}
The name \emph{$\left(\phi,\cK\right)$-property} is chosen to make a reference to the $D$-property introduced by Nesterenko in~\cite{N1996}. In the case when $\cK$ is a set of prime ideals and $\phi=D$ is a derivation our $\left(\phi,\cK\right)$ property is a weakening of $D$-property. Indeed, the $D$-property is as follows: we require the existence of a constant $C_1$ such that for every $D$-stable prime ideal $\idp$ one has
\begin{equation} \label{def_D_peoperty_classical}
\min_{P\in\idp}\ord_{\ull{f}}P(\ull{f})\leq C_1.
\end{equation}
It is easy to verify that $\min_{P\in\idp}\ord_{\ull{f}}P(\ull{f})\geq\ord_{\ull{f}}(\idp)$, hence the following property is weaker than~\eqref{def_D_peoperty_classical} (that is~\eqref{def_D_peoperty_classical} implies~\eqref{def_D_peoperty_modified}):
\begin{equation} \label{def_D_peoperty_modified}
\ord_{\ull{f}}(\idp)\leq C_1.
\end{equation}
In the inequality~\eqref{def_RelMinN2} above we have even weaker condition: the r.h.s. of~\eqref{def_RelMinN2} grows as grows the complexity of the ideal $\idq$.
\end{remark}

We mention here two technical lemmas that we shall use later, notably in the proof of Proposition~\ref{PropositionLdMprincipal}. Proofs are easy and can be found in~\cite{EZ2010}, Chapter~1.

Lemma~\ref{Representants} below provides us a possibility to replace the quantity $\Ord(X,Y)$, measuring the distance between two points $X$ and $Y$ in a projective space, by a quantity that is easier to control in the situation considered in the proof of Proposition~\ref{PropositionLdMprincipal}.

\begin{lemma} \label{Representants} Let $X, Y \in \mpp^n_{\overline{\kk((\b{z}))}}$ be two points in the projective space.

a) Let $\ull{x}\in\overline{\kk((\b{z}))}^{n+1}$ be a system of projective coordinates of $X$ and $\ull{y}\in\overline{\kk((\b{z}))}^{n+1}$ be a system of projective coordinates of $Y$ satisfying
\begin{equation*}
 \ordz \ull{x} = \ordz \ull{y}.
\end{equation*}
Then
\begin{equation} \label{LemmeRepresentantsA}
 \Ordz(X,Y) \geq \ordz(\ull{x} - \ull{y}) - \ordz\ull{y}.
\end{equation}

b) Suppose $\Ordz(X,Y)>0$, if we fix for $Y$ a system of projective coordinates $\ull{y}$ in $\overline{\kk((\b{z}))}^{n+1}$, then there is a system of projective coordinates $\ull{x} \in \overline{\kk((\b{z}))}^{n+1}$ of $X$ satisfying
\begin{equation}
 \begin{split}
  &\alpha) \quad \ordz\ull{x}=\ordz\ull{y},\\
  &\beta) \quad \Ordz(X,Y) = \ordz(\ull{x} - \ull{y}) - \ordz(\ull{y})
\end{split}
\end{equation}
\end{lemma}
\begin{proof}
See Lemma~1.22 of~\cite{EZ2010}.
\end{proof}

\begin{lemma}[Liouville's inequality] \label{lemma_Liouville_ie}
Let $Q\in\kk(\b{z})$ and $\b{Z}$ be a cycle in $\mpp_{\ol{\kk(z)}}^n$ of dimension 0 defined over $\kk(\b{z})$. Then
\begin{equation}\label{iet_main}
    \deg(\b{Q})h(\b{Z})+h(\b{Q})\deg(\b{Z})\geq\left|\sum_{\b{\beta}\in\b{Z}}\ordz\left(\b{Q}(\ull{\beta})\right)\right|,
\end{equation}
\end{lemma}
\begin{proof}
See inequality~(1.20) at the end of section~1.2.2 of~\cite{EZ2010}.
\end{proof}

In Definition~\ref{def_delta} here below we associate to each bi-projective ideal $I$ a couple of integers, $\left(\delta_0(I),\delta_1(I)\right)$. This quantity plays an important role in our article. It seems to be quite complicated at first glance, so we make a short intuitive comment to explain it in Remark~\ref{rem_def_delta} just after the definition.

Note that in Definition~\ref{def_delta} below we use constants $\mu,\nu_0$ and $\nu_1$. These constants are supposed to be the same as in the property~\eqref{degphiQleqdegQ}. In fact we shall use Definition~\ref{def_delta} only in situations when we have a map $\phi$ with a fixed choice of constants $\mu,\nu_0$ and $\nu_1$ to satisfy the property~\eqref{degphiQleqdegQ}. Hence this implicit dependence will not lead to any ambiguity.
\begin{definition} \label{def_delta}
\begin{enumerate}
\item Let $I,\idp \subset \AnneauDePolynomes$ be bi-homogeneous ideals, such that $I\not\subset\idp$. We choose a bi-homogeneous polynomial $P \in I \setminus \idp$
that minimizes the quantity
\begin{multline} \label{def_delta_condition_minimum_crossproduct}
\mu\dd_{(0,n-\rg I+1)}I\deg_{\ul{X}}P + \nu_0\dd_{(1,n-\rg I)}I \deg_{\ul{X}'}P \\ + \nu_1\dd_{(1,n-\rg I)}I \deg_{\ul{X}}P.
\end{multline}
If there exists more than one bi-homogeneous polynomial from $I\setminus\idp$ minimizing~(\ref{def_delta_condition_minimum_crossproduct})
we choose among them the polynomial with the minimal degree $\deg_{\ul{X}}P$.
We introduce notation $\delta_0(I,\idp) \eqdef \deg_{\ul{X}'}P$ and $\delta_1(I,\idp) \eqdef \deg_{\ul{X}}P$.
\item For all cycles $Z$ (defined over $\kk$) in $\mpp^1\times\mpp^n$ and such that $\I(Z)\not\subset\idp$
we define $\delta_i(Z,\idp) \eqdef \delta_i(\I(Z),\idp)$, $i=0,1$.
\item Let $\ul{f}=(1:\b{z},1:f_1(\b{z}):...:f_n(\b{z})) \in \mpp^1_{\kk[[\b{z}]]}\times\mpp^n_{\kk[[\b{z}]]}$. We recall Definition~\ref{def_tf}, where the bi-homogeneous ideal $\idp_{\ul{f}}$ is defined as a  bi-homogeneous ideal generated by polynomials $P\in\A$ vanishing at $\ul{f}$. We introduce the notations
$$
\delta_i(I,\ul{f}):=\delta_i(I,\idp_{\ul{f}})\text{ and } \delta_i(Z,\ul{f}):=\delta_i(Z,\idp_{\ul{f}})\text{ for } i=0,1,
$$
for all bi-homogeneous ideals $I$  such that $I\not\subset\idp_{\ul{f}}$ and all cycles $Z\subset\mpp^1\times\mpp^n$ such that $I(Z)\not\subset\idp_{\ul{f}}$.
\end{enumerate}
\end{definition}
\begin{remark} \label{rem_def_delta}
The quantities that we shall use in the subsequent considerations are eventually $\delta_i(I,\ul{f})$ and $\delta_i(Z,\ul{f})$, $i=0,1$.
Note that $\delta_i(I,\{0\})=\delta_i(I)$, $i=0,1$, in the sense of Definition~2.37 from~\cite{EZ2011}, hence if functions $f_1(z),\dots,f_n(z)$ are algebraically independent over $\kk(z)$ we have $\delta_i(I,\idp_{\ul{f}})=\delta_i(I)$, $i=0,1$ and we appear in the situation considered in~\cite{EZ2011}.

The quantities $\delta_i(I,\ul{f})$, $i=0,1$, plays in our proofs a role of an important characteristic of an ideal $I$. We refer the reader to~\cite{EZ2011}, Remark~2.38 for more detailed discussion on this matter. The modification that we have introduced in this work, compared to $\delta_i(I)$ in~\cite{EZ2011}, has the following reason. In our proofs we consider ideals generated by bi-homogeneous polynomials of degree comparable to $\delta_i(I,\ul{f})$, $i=0,1$ (for instance, see Definitions~\ref{V_irho_i} and ~\ref{def_i0} below). The essential part of the proof is the comparison for such ideals $I$ of their degrees $\deg I$ and the quantities $\ord_{\ul{f}}I$. Loosely speaking, we deduce the multiplicity estimate from the statement that in our construction the ideals of a bounded degree can't have an arbitrary big order of vanishing at $\ul{f}$. However we naturally have to exclude all the polynomials vanishing at $\ul{f}$, that is $\idp_{\ul{f}}$ (see Definition~\ref{def_tf}).

\end{remark}
Here is the first property of the quantity $\left(\delta_0(I,\ul{f}),\delta_1(I,\ul{f})\right)$.
\begin{lemma}\label{deltainfty} Fix a point
$$
\ul{f}=(1:\b{z},1:f_1(\b{z}):...:f_n(\b{z})) \in \mpp^1_{\kk[[\b{z}]]}\times\mpp^n_{\kk[[\b{z}]]}
$$
and consider a sequence of cycles $Z_i \subset \mpp^1\times\mpp^n$, $\in\mnn$, defined over $\kk$ and such that $\ul{f} \not\in Z_i$ for $i\in\mnn$. If $\ord_{\ul{f}}(Z_i)$ tends to $+\infty$ (as $i\rightarrow\infty$), then $\max\left(\delta_0(Z_i,\ul{f}),\delta_1(Z_i,\ul{f})\right)$ also tends to the infinity (as $i\rightarrow\infty$).
\end{lemma}
The proof of Lemma~\ref{deltainfty} is easy and can be found in~\cite{EZ2010} (Lemma~1.23). We shall need only its weak corollary:

\begin{corollary}\label{cor1_deltainfty} Let $z,f_1(z),\dots,f_n(z)\in\kk[[\b{z}]]$. There exists a constant $C_{sg}$ which depends only on $\ul{f}=(1:\b{z},1:f_1(\b{z}):...:f_n(\b{z})) \in \mpp^1_{\kk[[\b{z}]]}\times\mpp^n_{\kk[[\b{z}]]}$ such that if a cycle $Z \subset \mpp^1\times\mpp^n$ (defined over $\kk$) does not contain $\ul{f}$ and satisfies $\ord_{\ul{f}}Z \geq C_{sg}$, then either $\delta_0(Z,\ul{f})\geq\max(4,2n!+1,\frac{2\nu_1}{\max(\mu,\nu_0)})$ or $\delta_1(Z,\ul{f})\geq \max(2^n,4n)$.
\end{corollary}

The following definition is widely used in the subsequent considerations. 
\begin{definition} \label{V_irho_i}
\hidden{======================================
We define
\begin{equation}\label{def_nu}
    \nu\eqdef\begin{cases}1 &\text{ if } \nu_1=0,\\
                   2^{n+2}\max\left(1,\frac{4\nu_0}{\nu_1}\right)^{n+1} &\text{ if } \nu_1\ne 0.
             \end{cases}
\end{equation}
\begin{equation}\label{def_nu}
    \nu\eqdef 2^{n+2}\max\left(1,\frac{4\nu_0}{\nu_1}\right)^{n+1}.
\end{equation}
==========================================}
We define a sequence of numbers $\rho_i$ recursively. We put $\rho_0=0$, $\rho_1=1$ and 
$$
\rho_{i+1} = 6^{n+2}(n+2)^{(n+1)^2}\rho_i^{n+2}\max\left(\mu,\nu_0\right)^{6^{n+2}(n+2)^{(n+1)^2}\rho_i^{n+1}}
$$
for $i=1,...,n+1$. The constants $\mu$, $\nu_0$ and $\nu_1$ in this definition are the same as in~(\ref{degphiQleqdegQ}).

Let $Z$ be an algebraic bi-projective cycle defined over $\kk$ in the space $\mpp^1\times\mpp^n$. Let
$$
\ul{f}=(1:\b{z},1:f_1(\b{z}):...:f_n(\b{z})) \in \mpp^1_{\kk[[\b{z}]]}\times\mpp^n_{\kk[[\b{z}]]}.
$$
We denote by $V_i$, or more precisely by $V_i(Z,\ul{f})$, the vector space (over $\kk$) generated by $\idp_{\ul{f}}$ (see Definition~\ref{def_tf}) and the bi-homogeneous
polynomials from $\kk[X_0':X_1',X_0:...:X_n]$ vanishing over the cycle $Z$ and of degree in $\ul{X}'$ at most $\rho_i\left(\delta_0(Z,\ul{f})+\frac{\nu_1}{\max(\mu,\nu_0)}\delta_1(Z,\ul{f})\right)$ and of degree in $\ul{X}$ at most $\rho_i\delta_1(Z,\ul{f})$
(recall that $\delta_0(Z,\ul{f})$ and $\delta_1(Z,\ul{f})$ are introduced in Definition~\ref{def_delta}).

If $I$ is a proper bi-homogeneous ideal of $\AnneauDePolynomes$ we also use the notation
\begin{equation*}
    V_i(I):=V_i(\V(I)),
\end{equation*}
where $\V(I)$ is the cycle of $\mpp^1\times\mpp^n$ defined by $I$.
\end{definition}

\hidden{==========================
\begin{remark} \label{rem_V_irho_i}
Basically, $V_i$ represents a set of polynomials of bi-degree only a constant (i.e. $\rho_i$) times bigger than a "minimal" bi-degree (in terms explained in Definition~\ref{def_delta} and Remark~\ref{rem_def_delta}). Note that constants $\rho_i$ are absolute (they depend only on initial data: number of functions $n$ and constants $\mu$, $\nu_0$ and $\nu_1$ controlling the growth of degree of polynomial under the action of $\phi$).

\end{remark}
===============================}


\begin{definition} \label{def_i0} We associate to every bi-projective variety $W$ and a point
$$
\ul{f}=(1:\b{z},1:f_1(\b{z}):...:f_n(\b{z})) \in \mpp^1_{\kk[[\b{z}]]}\times\mpp^n_{\kk[[\b{z}]]}
$$
a number $i_0(W,\ul{f})$. We define $i_0(W,\ul{f})$ to be the biggest positive integer such that
$\rg\left(I\left(V_i(W,\ul{f}),\I(W)\right)\right)\geq i+r_{\ul{f}}$ for all $1\leq i\leq i_0(W)$.
\end{definition}


\begin{remark} \label{rem_i0}
In view of Definitions~\ref{def_delta} and~\ref{V_irho_i}, one readily verifies the inequality
$$
\rg\left(I\left(V_1(W,\ul{f}),\I(W)\right)\right)\geq 1+r_{\ul{f}}.
$$
So, the index $i_0(W,\ul{f})\geq 1+r_{\ul{f}}$ is well defined
for all varieties $W$. On the other hand the rank of any bi-homogeneous ideal in $\AnneauDePolynomes$ can not exceed $n+1$, thus $i_0(W,\ul{f})\leq n+1$ for every variety $W\subset\mpp^1\times\mpp^n$.
\end{remark}


The lemmas below represent two important ingredients of the proof of our main result.

The proof of Lemma~\ref{LemmeProp13} can be found in~\cite{EZ2011},~\S3 or  in~\cite{EZ2010}, subsection~2.2.2.
\begin{lemma}\label{LemmeProp13}
Let $\idp$ be a prime ideal of $\A$
such that the map $\phi$ is correct with respect to this ideal and
let $V \subset \A$, $V\ne\{0\}$, be a $\kk$-linear subspace of
$\A$. If $e_{\phi}(V,\idp)>m(I(V,\idp))$, then there
exists an equidimensional $\phi$-stable ideal $J$ such that
\begin{list}{\alph{tmpabcdmine})}{\usecounter{tmpabcdmine}}
 \item \label{LemmeProp13_a} $I(V,\idp) \subset J \subset \idp$, 
 \item \label{LemmeProp13_b} $\rg(J)=\rg(I(V,\idp))$, 
 \item \label{LemmeProp13_c} all the primes associated to $J$ are contained in $\idp$. 
\end{list}
In particular,
\begin{equation} \label{LemmeProp13_en_part}
\begin{aligned}
    &m(J) \leq m(I(V,\idp)),\\
    &\dd_{(1, n-\rg J)} J \leq \dd_{(1, n-\rg I(V,\idp))}(I(V,\idp)),\\
    &\dd_{(0, n-\rg J+1)} J \leq \dd_{(0, n-\rg I(V,\idp)+1)}(I(V,\idp)).
\end{aligned}
\end{equation}
\end{lemma}

Lemma~\ref{LemmeCor14NumberW} below is an analogue of Lemmas~2.18 and~2.21 of~\cite{EZ2010}, or also of Lemma~2.44 of~\cite{EZ2011}. We prove this lemma in the next section.
\begin{lemma} \label{LemmeCor14NumberW} 
Let
$$
\ul{f}=(1:\b{z},1:f_1(\b{z}):...:f_n(\b{z})) \in \mpp^1_{\kk[[\b{z}]]}\times\mpp^n_{\kk[[\b{z}]]}.
$$
Let $\idp\subset\AnneauDePolynomes$ be a prime bi-homogeneous ideal such that
$\idp_{\ul{f}}\subsetneq\idp$ and $\V(\idp)$ is projected onto $\mpp^1$.
We recall the notations 
$V_i=V_i(\idp,\ul{f})$ and $\rho_i$ introduced in
Definition~\ref{V_irho_i} and $I(V_i,\idp)=(V_i\A_{\idp})\cap\A$ introduced in Definition~\ref{def_I}. Assume that either $\delta_0\geq \max\left(2,\frac{2\nu_1}{\max(\mu,\nu_0)}\right)$ or $\delta_1(\idp,\ul{f})\geq 2^n$. One has the following upper bound for $m(I(V_i,\idp))$:
\begin{equation} \label{majorationm}
m(I(V_i,\idp)) \leq 6^{n+2}(n+2)^{(n+1)(n-t_{\ul{f}}+1)}\rho_i^{t_{\ul{f}}+1}.
\end{equation}
\begin{definition} \label{def_Cm}
We introduce the following notation:
\begin{equation}
C_m:=6^{n+2}(n+2)^{(n+1)^2}\rho_{n+1}^{n+1}.
\end{equation}
So, $C_m$ is the upper bound for the r.h.s. of~\eqref{majorationm}. Note that $C_m$ depends on $n$ only.
\end{definition}
\end{lemma}

\subsection{Proof of Lemma~\ref{LemmeCor14NumberW} 
} \label{section_proof_LemmeCor14NumberW_case_nu1_eq_zero}
We recall that the ideal $\idp_{\ul{f}}$, its rank $r_{\ul{f}}$ and the transcendence degree $t_{\ul{f}}$ are introduced in Definition~\ref{def_tf}. In this section we use the following notation (see~\cite{PP}, p.12)
\begin{equation} \label{def_deg_X_ab}
\deg(X,a,b):=\deg_{(0,\dim X)}(X)\cdot b^{\dim X}+\dim X\cdot\deg_{(1,\dim X-1)}(X)\cdot a\cdot b^{\dim X-1}
\end{equation}
where $X\subset\mpp^1\times\mpp^n$ is a variety and $a,b\in\mrr^+$.

\begin{lemma} \label{LemmeProp14_1_general} Let $I$ be a bi-homogeneous ideal of $\AnneauDePolynomes$, $I\ne\AnneauDePolynomes$ and
$$
\ul{f}=(1:\b{z},1:f_1(\b{z}):...:f_n(\b{z})) \in \mpp^1_{\kk[[\b{z}]]}\times\mpp^n_{\kk[[\b{z}]]}.
$$
We denote $\delta_0:=\delta_0(I,\ul{f})$, $\delta_1=\delta_1(I,\ul{f})$ (recall that the quantities $\delta_0(I,\ul{f})$ and $\delta_1(I,\ul{f})$ are introduced in Definition~\ref{def_delta}).

Let $W \subsetneq \V(\idp_{\ul{f}})$ be an irreducible (bi-projective) variety
 projecting onto the factor $\mpp^1$ and let positive integers $a,b\in\mnn$ satisfy
\begin{multline} \label{LemmeProp14_ie1_1}
\mu\cdot b\cdot\dd_{(0,\dim I)} I + \nu_0\cdot a\cdot\dd_{(1,\dim I-1)}I + \nu_1\cdot b\cdot\dd_{(1,\dim I-1)}I\\<\mu\delta_1\dd_{(0,\dim I)} I + \nu_0\delta_0\dd_{(1,\dim I-1)}I
 + \nu_1\delta_1\dd_{(1,\dim I-1)}I
\end{multline}
or
\begin{multline} \label{LemmeProp14_ie1_1_2}
\mu\cdot b\cdot\dd_{(0,\dim I)} I + \nu_0\cdot a\cdot\dd_{(1,\dim I-1)}I + \nu_1\cdot b\cdot\dd_{(1,\dim I-1)}I\\=\mu\delta_1\dd_{(0,\dim I)} I + \nu_0\delta_0\dd_{(1,\dim I-1)}I
 + \nu_1\delta_1\dd_{(1,\dim I-1)}I
\end{multline}
and
\begin{equation} \label{LemmeProp14_ie1_1_2part2}
b<\delta_1.
\end{equation}

Assume
\begin{equation} \label{LemmeProp14_hypothese1_1}
\deg(W,a,b) + \dim(W) \\< (t_{\ul{f}}+2)2^{-n-1}(n+2)^{-(n+1)(n-t_{\ul{f}}+1)}\deg\left(\V(\idp),a,b\right).
\end{equation}
Then, there is a polynomial $Q \in \I(W) \setminus (I\cup\idp_{\ul{f}})$ (that is, $Q$ vanishes on $W$, does not belong to $I$ and does not vanish at $\ul{f}$) satisfying two inequalities:
\begin{equation} \label{LemmeProp14_petitPolynome_1_case_b}
\begin{aligned}
\deg_{\ul{X}'}Q&\leq a,\\
\deg_{\ul{X}}Q&\leq b.
\end{aligned}
\end{equation}
\end{lemma}
\begin{proof}
Note that the assumption that $W$ projects onto $\mpp^1$ implies $1\leq\dim W$, and the assumption $W \subsetneq \V(\idp_{\ul{f}})$ implies $\dim W<t_{\ul{f}}+1$. 


Suppose first that no polynomial of bi-degree at most $(a,b)$ simultaneously vanishes on $W$ and does not vanish at $\ul{f}$. In other terms, suppose that the $\kk$-linear space of polynomials vanishing on $W$ and of bi-degree upper bounded by $(a,b)$ is included\footnote{Moreover, the assumption $W\subsetneq\V(\idp_{\ul{f}})$ implies the inclusion in the opposite direction, $\idp_{\ul{f}}\subset\V(W)$, so these two $\kk$-linear spaces in fact coincide in the case under consideration. However this is not important for our proof. We mentioned this fact just to show that the lower bound in~\eqref{LemmeProp14_Hgminoration_1_preliminary} is in fact an equality.} in the $\kk$-linear space of polynomials of bi-degree at most $(a,b)$ and belonging to the ideal $\idp_{\ul{f}}$. Hence
\begin{equation} \label{LemmeProp14_Hgminoration_1_preliminary}
H_g(W,a,b)\geq H_g\left(\V(\idp),a,b\right).
\end{equation}
Using Corollary~9 of~\cite{PP} we continue~\eqref{LemmeProp14_Hgminoration_1_preliminary}
\begin{equation} \label{LemmeProp14_Hgminoration_1}
H_g(W,a,b)\geq H_g\left(\V(\idp),a,b\right) \geq
(t_{\ul{f}}+2)2^{-n-1}(n+2)^{-(n+1)(n-t_{\ul{f}}+1)}\deg\left(\V(\idp),a,b\right).
\end{equation}

At the same time, by another part of Corollary~9 of~\cite{PP} we have
\begin{equation} \label{LemmeProp14_Hgmajoration_1}
H_g(W,a,b) \leq \deg(W,a,b)+ \dim W.
\end{equation}

The system of inequalities~\eqref{LemmeProp14_Hgminoration_1} and~\eqref{LemmeProp14_Hgmajoration_1} 
contradicts the assumption~\eqref{LemmeProp14_hypothese1_1}. We conclude that there exists a bi-homogeneous polynomial $Q_0$ of bi-degree at most $(a,b)$ that vanishes on $W$ and does not vanish at $\ul{f}$. Moreover, we claim that $Q_0$ does not belong to $I$. Indeed, it follows from assumptions~\eqref{LemmeProp14_ie1_1}, \eqref{LemmeProp14_ie1_1_2} and \eqref{LemmeProp14_ie1_1_2part2}, because the r.h.s. of~\eqref{LemmeProp14_ie1_1} minimizes the expression
\begin{equation*}
    \mu\deg_{\ul{X}}(Q)\dd_{(0,\dim I)} I + \nu_0\deg_{\ul{X}'}(Q)\dd_{(1,\dim I-1)}I + \nu_1\deg_{\ul{X}}(Q)\dd_{(1,\dim I-1)}I
\end{equation*}
for all the bi-homogeneous polynomials $Q$ from $I\setminus\{\idp_{\ul{f}}\}$ (see Definition~\ref{def_delta}). Lemma~\ref{LemmeProp14_1_general} is proved.
\end{proof}


\begin{corollary} \label{LemmeCor14_1}
Consider the situation of Lemma~\ref{LemmeProp14_1_general}. So, let $I$ be a bi-homogeneous ideal of $\AnneauDePolynomes$, $I\ne\AnneauDePolynomes$ and
$$
\ul{f}=(1:\b{z},1:f_1(\b{z}):...:f_n(\b{z})) \in \mpp^1_{\kk[[\b{z}]]}\times\mpp^n_{\kk[[\b{z}]]}.
$$
We denote $\delta_0:=\delta_0(I,\ul{f})$, $\delta_1=\delta_1(I,\ul{f})$ (recall that the quantities $\delta_0(I,\ul{f})$ and $\delta_1(I,\ul{f})$ are introduced in Definition~\ref{def_delta}). Let $W \subsetneq \V(\idp_{\ul{f}})$ be an irreducible (bi-projective) variety  projecting onto the factor $\mpp^1$

Assume in addition that $I$ is a radical ideal and assume that the variety $W$  contains $\V(I)$. Moreover, assume that $\V(I)$ itself projects onto the factor $\mpp^1$. Assume $\dim(W)\geq 2$ and assume that either $\delta_0\geq 2$ or $\delta_1\geq 2^n$.
Then, if $\nu_1=0$,
\begin{equation} \label{minorationdegW_1}
\deg\left(W,\delta_0,\delta_1\right) \\ \geq (t_{\ul{f}}+2)2^{-n-2}(n+2)^{-(n+1)(n-t_{\ul{f}}+1)}\deg\left(\V(\idp_{\ul{f}}),\delta_0,\delta_1\right).
\end{equation}
\end{corollary}
\begin{proof}
In the beginning, we consider the problem with an auxiliary assumption
\begin{equation} \label{LemmeCor14_1_assumption_one}
\delta_0\geq 1\text{ and }\delta_1\geq 1.
\end{equation}
By hypothesis, we have either $\delta_0\geq 2$ or $\delta_1\geq 2^n$. If $\delta_0\geq 2$, we apply Lemma~\ref{LemmeProp14_1_general} with $a=\delta_0-1$ and $b=\delta_1$. Otherwise, we necessarily have $\delta_1\geq 2^n$ and we apply Lemma~\ref{LemmeProp14_1_general} with $a=\delta_0$ and $b=\delta_1-1$. Clearly, condition~\eqref{LemmeProp14_ie1_1} is satisfied in both cases (note that $\deg_{(1,\dim I-1)}I\geq 1$ in view of our assumption that $\V(I)$ projects onto the factor $\mpp_1$).

In the following text we assume $\delta_0\geq 2$. The case $\delta_1\geq 2^n$ can be treated in exactly the same way, and it is left to the interested reader as an exercise.

As we assume $\V(I)\subset W$, the conclusion of Lemma~\ref{LemmeProp14_1_general} can not hold true and we infer that assumption~\eqref{LemmeProp14_hypothese1_1} has to fail, that is we have
\begin{multline} \label{minorationdegW_1_1}
\deg\left(W,\delta_0-1,\delta_1\right)+\dim(W) \\ \geq (t_{\ul{f}}+2)2^{-n-1}(n+2)^{-(n+1)(n-t_{\ul{f}}+1)}\deg\left(\V(\idp_{\ul{f}}),\delta_0-1,\delta_1\right).
\end{multline}
Because of our assumption that $W$ projects onto $\mpp_1$, we have
\begin{eqnarray*}
\dim(W)&\geq& 1,\\
\deg_{(1,\dim(W)-1)}(W)&\geq& 1,
\end{eqnarray*}
hence
\begin{equation*}
\begin{aligned}
\deg&\left(W,\delta_0-1,\delta_1\right)+\dim(W)=\deg_{(0,\dim W)}(W)\cdot \delta_1^{\dim W}
\\
&\qquad+\dim W\cdot\deg_{(1,\dim W-1)}(W)\cdot (\delta_0-1)\cdot \delta_1^{\dim W-1}+\dim(W)\\
&\leq \deg_{(0,\dim W)}(W)\cdot \delta_1^{\dim W}
\\
&\qquad+\dim W\cdot\deg_{(1,\dim W-1)}(W)\cdot \delta_0\cdot \delta_1^{\dim W-1}
\\
&\leq\dim(W,\delta_0,\delta_1).
\end{aligned}
\end{equation*}
At the same time, in view of our assumption $\delta_0\geq 2$ we have
$$
\deg\left(\V(\idp_{\ul{f}}),\delta_0-1,\delta_1\right)\geq\frac12\deg\left(\V(\idp_{\ul{f}}),\delta_0,\delta_1\right)
$$
and we readily deduce~\eqref{minorationdegW_1}.

Now assume that~\eqref{LemmeCor14_1_assumption_one} does not hold true. If $\delta_1=0$, then the r.h.s. of~\eqref{minorationdegW_1} is zero and the claim readily follows. It remains us to consider the case $\delta_0=0$ and $\delta_1\geq 2^n$. In this case, we apply Lemma~\ref{LemmeProp14_1_general} with $a=\delta_0=0$ and $b=\delta_1-1$. We infer
\begin{multline} \label{minorationdegW_1_1_final}
\deg\left(W,\delta_0,\delta_1-1\right)+\dim(W) \\ \geq (t_{\ul{f}}+2)2^{-n-1}(n+2)^{-(n+1)(n-t_{\ul{f}}+1)}\deg\left(\V(\idp_{\ul{f}}),\delta_0,\delta_1-1\right).
\end{multline}
Further, $\delta_1\geq 2^n$ implies
$$
\deg\left(W,\delta_0,\delta_1-1\right)+\dim(W) \leq \deg\left(W,\delta_0,\delta_1\right),
$$
and
$$
\deg\left(\V(\idp_{\ul{f}}),\delta_0,\delta_1-1\right)\geq\frac12\deg\left(\V(\idp_{\ul{f}}),\delta_0,\delta_1\right).
$$
The conclusion~\eqref{minorationdegW_1} readily follows.
\end{proof}
\begin{corollary} \label{LemmeCor14_1_nu_one_positive}
In the situation of Corollary~\ref{LemmeCor14_1}, replace the hypothesis $\nu_1=0$ by the hypothesis $\nu_1>1$.  Further,  introduce the following hypothesis on $\delta_0$ and $\delta_1$:  either $\delta_0\geq \max\left(2,\frac{2\nu_1}{\max(\mu,\nu_0)}\right)$ or $\delta_1\geq 2^n$. Then we have the following inequality:
\begin{multline} \label{minorationdegW_1_nu_one_positive}
\deg\left(W,\delta_0+\frac{\nu_1}{\max(\mu,\nu_0)}\delta_1,\delta_1\right) \\ \geq (t_{\ul{f}}+2)6^{-n-2}(n+2)^{-(n+1)(n-t_{\ul{f}}+1)}
\deg\left(\V(\idp_{\ul{f}}),\delta_0+\frac{\nu_1}{\max(\mu,\nu_0)}\delta_1,\delta_1\right).
\end{multline}
\end{corollary}
\begin{proof}
The proof of this corollary is similar to the proof of Corollary~\eqref{LemmeCor14_1}. For the purpose of not to waste space on the trivial calculations, we consider here only the case $\delta_0\geq 2$ and $\delta_1\geq 2$, leaving details for the reader.

Note that the numbers $a=\left[\delta_0/2+\frac{\nu_1}{2\max(\mu,\nu_0)}\delta_1\right]$ and $b=\left[\delta_1/2\right]$ satisfy the hypothesis~\eqref{LemmeProp14_ie1_1}. Hence we deduce with Lemma~\ref{LemmeProp14_1_general}
\begin{multline} \label{minorationdegW_1_nu_one_positive_ie_one}
\deg\left(W,\left[\delta_0/2+\frac{\nu_1}{2\max(\mu,\nu_0)}\delta_1\right],\left[\delta_1/2\right]\right) \\ \geq (t_{\ul{f}}+2)2^{-n-2}(n+2)^{-(n+1)(n-t_{\ul{f}}+1)}\\\times\deg\left(\V(\idp_{\ul{f}}),\left[\delta_0/2+\frac{\nu_1}{2\max(\mu,\nu_0)}\delta_1\right],\left[\delta_1/2\right]\right),
\end{multline}
where $[x]$ denotes the integer part of a real $x$, that is the biggest integer $n\in\mzz$ satisfying $n\leq x$.
Using the inequality $[r/2]\geq r/3$ for every real $r\geq 2$, we readily deduce
\begin{multline} \label{minorationdegW_1_nu_one_positive_ie_two}
\deg\left(W,\delta_0/2+\frac{\nu_1}{2\max(\mu,\nu_0)}\delta_1,\delta_1/2\right) \\ \geq (t_{\ul{f}}+2)2^{-n-2}(n+2)^{-(n+1)(n-t_{\ul{f}}+1)}\\\times\deg\left(\V(\idp_{\ul{f}}),\delta_0/3+\frac{\nu_1}{3\max(\mu,\nu_0)}\delta_1,\delta_1/3\right).
\end{multline}
Finally, definition~\eqref{def_deg_X_ab} readily implies
\begin{equation} \label{minorationdegW_1_nu_one_positive_deg_X_a_b}
\deg(X,\lambda a,\lambda b)=\lambda^{\dim X}\deg(X,a,b)
\end{equation}
for any variety $X\in\mpp^1\times\mpp^n$ and $a,b,\lambda\in\mrr^+$. Applying~\eqref{minorationdegW_1_nu_one_positive_deg_X_a_b} to the both sides of~\eqref{minorationdegW_1_nu_one_positive_ie_two} we deduce~\eqref{majorationm}.
\end{proof}
\begin{proof}[Proof of Lemma~\ref{LemmeCor14NumberW} 
]
We denote by $r$ the rank $\rg I(V_i,\idp)$ and by $\delta_i$ the quantities $\delta_i(\idp,\ul{f})$, $i=0,1$.

To start with, consider the case $\nu_1=0$. Then the ideal $I(V_i,\idp)=V_i\A_{\idp}\cap\A$ is extended-contracted by localization at $\idp$ of an ideal generated by $\idp_{\ul{f}}$ and by $r-r_{\ul{f}}$
polynomials of bi-degree upper bounded by $\left(\rho_i\delta_0,\rho_i\delta_1\right)$ (see Definition~\ref{def_i0}). Using B\'ezout's theorem (see Lemma~\ref{lemma_BT}) and definition~\eqref{def_deg_X_ab} we readily verify
\begin{equation} \label{majorationTordue0_1}
\begin{aligned}
\deg\left(\V(I(V_i,\idp)),\delta_0,\delta_1\right)\leq\rho_i^{r-r_{\ul{f}}}\deg\left(\V(\idp_{\ul{f}}),\delta_0,\delta_1\right).
\end{aligned}
\end{equation}

Let $W=\V(\idq)$, where $\idq$ is a minimal prime ideal associated to $\V(I(V_i,\idp))$. By construction of $I(V_i,\idp)$
all its associated primes are contained in $\idp$, thus $\V(\idp) \subset W$. Moreover, since
$\V(\idp)$ is projected onto $\mpp^1$, we deduce that $W$ is projected onto $\mpp^1$ as well. Further, by construction of $I(V_i,\idp)$ we have $\dim\V\left(I(V_i,\idp)\right)\geq 1$, and if $\dim\V\left(I(V_i,\idp)\right)=1$ then $I(V_i,\idp)=\idp$, hence $m\left(I(V_i,\idp)\right)=1$ and~\eqref{majorationm} follows. Thus we have to consider only the case $\dim(W)=\dim\V\left(I(V_i,\idp)\right)\geq 2$, thus we are in measure to apply Corollary~\ref{LemmeCor14_1}. We find that $W$
satisfies~(\ref{minorationdegW_1}); in other words, every prime $\idq$ associated to $I(V_i,\idp)$ satisfies
\begin{equation} \label{minorationDegQ_1}
\deg\left(\V(\idq),\delta_0,\delta_1\right) \\ \geq (t_{\ul{f}}+2)2^{-n-2}(n+2)^{-(n+1)(n-t_{\ul{f}}+1)}\deg\left(\V(\idp_{\ul{f}}),\delta_0,\delta_1\right).
\end{equation}
But
\begin{equation} \label{degDecomposition_1}
\begin{aligned}
 \dd_{(1,n-r)}(I(V_i,\idp)) = \sum_{\substack{\idq \in \Spec\A,\\ \rg(\idq)=r}}\dd_{(1,n-r)}(\idq)l_{\AnneauDePolynomes_{\idq}}(\left(\AnneauDePolynomes/I(V_i,\idp)\right)_{\idq}),
\end{aligned}
\end{equation}
and
\begin{equation} \label{htDecomposition_1}
\begin{aligned}
 \dd_{(0,n-r+1)}(I(V_i,\idp)) = \sum_{\substack{\idq \in \Spec\A,\\ \rg(\idq)=r}}\dd_{(0,n-r+1)}(\idq)l_{\AnneauDePolynomes_{\idq}}(\left(\AnneauDePolynomes/I(V_i,\idp)\right)_{\idq}),
\end{aligned}
\end{equation}
Summing up~(\ref{degDecomposition_1}) with coefficient $\delta_0\delta_1^{n-r}$
and~(\ref{htDecomposition_1}) with coefficient $\delta_1^{n+1-r}$, we find
\begin{equation}\label{minorationDegHtRightParm_1}
 \deg\left(I(V_i,\idp),\delta_0,\delta_1\right)=\sum_{\substack{\idq \in \Spec\A,\\ \rg(\idq)=r}}\deg(\idq,\delta_0,\delta_1)l_{\AnneauDePolynomes_{\idq}}(\left(\AnneauDePolynomes/I(V_i,\idp)\right)_{\idq})
\end{equation}
Applying~\eqref{majorationTordue0_1} to the l.h.s. of~\eqref{minorationDegHtRightParm_1} and~\eqref{minorationDegQ_1} to the r.h.s. of~\eqref{minorationDegHtRightParm_1} we obtain
\begin{multline} \label{majorationm_pre_1}
\rho_i^{r-r_{\ul{f}}}\deg\left(\V(\idp_{\ul{f}}),\delta_0,\delta_1\right)
\\
\geq
(t_{\ul{f}}+2)2^{-n-2}(n+2)^{-(n+1)(n-t_{\ul{f}}+1)}\deg\left(\V(\idp_{\ul{f}}),\delta_0,\delta_1\right)m(I(V_i,\idp)).
\end{multline}
Finally, we deduce~\eqref{majorationm} from~\eqref{majorationm_pre_1} with simplification and using the remark $0\leq t_{\ul{f}}\leq n$, $0\leq r \leq n+1$ (in fact, in this case, $\nu_1=0$, we obtain even a better constant in the r.h.s. of~\eqref{majorationm}).

In the case $\nu_1>0$ we proceed in the similar way. In this case the ideal $I(V_i,\idp)=V_i\A_{\idp}\cap\A$ is generated by $\idp_{\ul{f}}$ and by $r-r_{\ul{f}}$
polynomials of bi-degree upper bounded by $\left(\rho_i\left(\delta_0+\frac{\nu_1}{\max(\mu,\nu_0)}\delta_1\right),\rho_i\delta_1\right)$ (see Definition~\ref{def_i0}). Again, using B\'ezout's theorem (see Lemma~\ref{lemma_BT}) and definition~\eqref{def_deg_X_ab} we find
\begin{multline} \label{majorationTordue0_1}
\deg\left(\V(I(V_i,\idp)),\delta_0+\frac{\nu_1}{\max(\mu,\nu_0)}\delta_1,\delta_1\right)
\\
\leq\rho_i^{r-r_{\ul{f}}}\deg\left(\V(\idp_{\ul{f}}),\delta_0+\frac{\nu_1}{\max(\mu,\nu_0)}\delta_1,\delta_1\right).
\end{multline}
We consider a minimal prime ideal $\idq$ associated to $\V(I(V_i,\idp))$. We readily verify that we can apply Corollary~\ref{LemmeCor14_1_nu_one_positive} (see the first part of this proof for more details). We deduce with Corollary~\ref{LemmeCor14_1_nu_one_positive} the lower bound
\begin{multline} \label{Main_Lemma_minorationdegW_1_nu_one_positive}
\deg\left(\V(\idq),\delta_0+\frac{\nu_1}{\max(\mu,\nu_0)}\delta_1,\delta_1\right) \\ \geq (t_{\ul{f}}+2)6^{-n-2}(n+2)^{-(n+1)(n-t_{\ul{f}}+1)}\\\times\deg\left(\V(\idp_{\ul{f}}),\delta_0+\frac{\nu_1}{\max(\mu,\nu_0)}\delta_1,\delta_1\right).
\end{multline}
Using formulae~\eqref{degDecomposition_1} and~\eqref{htDecomposition_1} we find
\begin{multline} \label{majorationm_pre_1_part_two}
\rho_i^{r-r_{\ul{f}}}\deg\left(\V(\idp_{\ul{f}}),\delta_0+\frac{\nu_1}{\max(\mu,\nu_0)}\delta_1,\delta_1\right)
\\
\geq
(t_{\ul{f}}+2)6^{-n-2}(n+2)^{-(n+1)(n-t_{\ul{f}}+1)}
\\
\times\deg\left(\V(\idp_{\ul{f}}),\delta_0+\frac{\nu_1}{\max(\mu,\nu_0)}\delta_1,\delta_1\right)m(I(V_i,\idp)).
\end{multline}
Finally, we deduce~\eqref{majorationm} from~\eqref{majorationm_pre_1_part_two} with simplification and remark $0\leq t_{\ul{f}}\leq n$, $0\leq r \leq n+1$.
\end{proof}


\section{Transference lemma of P. Philippon} \label{section_transference_lemma}

From here on we assume that $t_{\ul{f}}\geq 1$. In the case $t_{\ul{f}}=0$, all the functions $f_1(z),\dots,f_n(z)$ are algebraic, hence multiplicity estimate follows immediately from Liouville's inequality (see Lemma~\ref{lemma_Liouville_ie}).

In this section we present (a particular case of) the transference lemma elaborated in~\cite{PP}. In the sequel we shall use the constant $c_n$ defined by
\begin{equation} \label{def_cn}
c_n=2^{n+1}(n+2)^{(n+1)(n+3)}.
\end{equation}
Note that in terms of~\cite{PP} one has $c_n=c_{\mpp^1\times\mpp^n}$.


\begin{theorem}[Transference $(1,n)$-Projective Lemma]\label{LdT}
Let $\ull{f} \in \mpp^n_{\kk[[\b{z}]]}$. We denote $t=t_{\ul{f}}$. Let $C$ be a real number satisfying 
\begin{eqnarray} \label{LdT_Cestgrande}
C&\geq&\left(c_n\right)^{t}\left(C_{\ul{f}}\right)^{t+1}\max\left(1,\nu_0,\mu\right)^{-t-1},\\
C&\geq&\left(\left(h(\idpf)+\deg(\idpf)\right)\max\left(1,\frac1{\nu_0},\frac1{\mu}\right)\right)^{1/(t+1)}. \label{LdT_Cestgrande_2}
\end{eqnarray}
If a homogeneous polynomial $\b{P} \in \kk[\b{z}][X_0,X_1,...,X_n]\setminus\idpf$ satisfies
\begin{multline} \label{ordPplusqueb}
 \ordz (\b{P}(\ull{f})) - \deg\b{P}\cdot\ordz (\ull{f}) - h(\b{P})\\ > C\cdot t\cdot\left((\nu_0+\mu) \left(h(\b{P})+1\right)+(\nu_1+\mu)\deg P\right)\mu^{t-1}\left(\deg\b{P}+1\right)^{t},
\end{multline}
then there is an  irreducible cycle $\b{Z} \in \mpp_n\left(\overline{\kk(\b{z})}\right)$ defined over $\kk(\b{z})$, of
dimension~0, contained in the zero locus of $\b{P}$ and in the zero locus of the ideal $\idp_{\ul{f}}$, satisfying
\begin{multline} \label{LdTdegZ}
 \nu_0\deg\b{Z}\cdot h(\b{P})+\nu_1\deg\b{Z}\cdot\deg\b{P}+\mu\cdot h(\b{Z})\cdot\deg\b{P}\\ \leq (c_n C)^{\frac{t-1}{t+1}}\mu^t\left(\deg P+1\right)^t\left( h(\idpf)\deg(P) + \deg(\idpf)h(P)\right),
\end{multline}
and
\begin{equation} \label{LdTordZ}
\sum_{\alpha \in \b{Z}} \Ord_{\ull{f}}(\alpha) > C^{\frac{1}{t+1}}c_n^{-\frac{t}{t+1}}\Big( \nu_0\deg(\b{Z})h(\b{P})+\nu_1\deg(\b{Z})\deg\b{P}\\+ \mu\cdot h(\b{Z})\deg\b{P} \Big).
\end{equation}
In particular, 
(\ref{LdTordZ}) implies
\begin{equation} \label{LdTordZ0}
\sum_{\ul{\alpha} \in \b{Z}} \Ord_{\ull{f}}(\ul{\alpha}) > C_{\ul{f}}\Big( \nu_0\deg(\b{Z})h(\b{P}) + \nu_1\deg(\b{Z})\deg\b{P} \\ + \mu\cdot h(\b{Z})\deg\b{P} \Big).
\end{equation}
\end{theorem}
\begin{proof}
We denote by $I_0$ the ideal corresponding to the intersection of $\V(\idpf)$ and $\V(P)$, i.e. the ideal given by  Proposition~4.11 of Chapter~3, \cite{NP}. By this proposition, the ideal $I_0$ satisfies
\begin{eqnarray}
\deg I_0 &\leq& \deg\idpf\cdot\deg P, \label{LdT_bound_degI0}\\
h(I_0) &\leq& h(\idpf)\cdot\deg P+\deg\idpf \cdot h(P), \label{LdT_bound_hI0}
\end{eqnarray}
and
\begin{multline} \label{LdT_lb_ordI0}
\Ord_{\ul{f}} I_0 \geq  \Ordz (\b{P}(\ull{f})) - \deg P\cdot h(\idpf) - h(P)\cdot\deg\idpf
\\
\geq  C\cdot t\cdot\left((\mu+\nu_0) \left(h(\b{P})+1\right)+(\nu_1+\mu)\deg\b{P}\right)\mu^{t-1}\left(\deg\b{P}+1\right)^{t}
\\
-\deg P\cdot h(\idpf) - h(P)\cdot\deg\idpf,
\end{multline}
the second inequality here is implied by the hypothesis~\eqref{ordPplusqueb}.

We can consider $(1:\b{z})$ as coordinates of a point in $\mpp^1_{\overline{\kk(\b{z})}}$ (see Remark~\ref{genZP}). Under this convention, the cycle $\tilde{X}_0:=\V(I_0)$ can be considered as a cycle in $\mpp^1\times\mpp^n$ of dimension $t$ and it has the following degrees:
\begin{eqnarray*}
\deg_{(0,t)}\tilde{X}_0=h(I_0),\\
\deg_{(1,t-1)}\tilde{X}_0=\deg I_0.
\end{eqnarray*}

We apply Corollary~11 of~\cite{PP} to $\tilde{\Phi}=\ullt{f}$ and
$\tilde{X}_0=\V(I_0)\subset\mpp^1\times\mpp^n$. We choose the multi-degree $(\eta,\delta)$ to be
\begin{eqnarray*}
  \eta &=& \left[(c_nC)^{\frac{1}{t+1}}\left(\nu_0(h(\b{P})+1)+\nu_1\deg\b{P}\right)\right], \\
  \delta &=& \left[(c_nC)^{\frac{1}{t+1}}\mu(\deg\b{P}+1)\right].
\end{eqnarray*}

Inequalities~\eqref{LdT_Cestgrande} and~\eqref{LdT_Cestgrande_2} imply
\begin{equation} \label{LdT_preuve_majoration_hdeg}
\begin{aligned}
  h(\b{P}) &\leq\eta, \\
   \deg\b{P} &\leq\delta,\\
   \max(c_n,C_{\ul{f}}) &\leq \min\left(\eta,\delta\right)
\end{aligned}
\end{equation}
(recall that the constant $C_{\ul{f}}$ is introduced in Definition~\ref{def_tf}), hence $\tilde{X}_0$ is defined by forms of multidegree $\leq(\eta,\delta)$ with
$$
\min(\eta,\delta)\geq c_n=c_{\mpp^1\times\mpp^n},
$$
where the constant $c_{\mpp^1\times\mpp^n}$ is the one defined in~\cite{PP}.

In our case we have
\begin{multline} \label{LdT_preuve_degX0etadelta}
\deg\left(\tilde{X}_0,\eta,\delta\right)=h(I_0)\cdot\delta^{t}+t\cdot\deg I_0\cdot\eta\delta^{t-1}\\
\leq\left( h(\idpf)\cdot\deg P+\deg\idpf \cdot h(P)\right)\cdot\delta^{t}+t\cdot\deg\idpf\cdot\deg P\cdot\eta\delta^{t-1}\\
\leq (c_nC)^{\frac{t}{t+1}}\Bigg(\left( h(\idpf)\cdot\deg P+\deg\idpf \cdot h(P)\right)\cdot\mu^t(\deg P+1)^t\\+t\cdot\deg\idpf\cdot\deg P\cdot\left(\nu_0(h(\b{P})+1)+\nu_1\deg\b{P}\right)\mu^{t-1}(\deg P+1)^{t-1}\Bigg)\\
=(c_nC)^{\frac{t}{t+1}}\mu^{t-1}(\deg P+1)^{t-1}\Bigg(\left( h(\idpf)\cdot\deg P+\deg\idpf \cdot h(P)\right)\cdot\mu(\deg P+1)\\+t\cdot\deg\idpf\cdot\deg P\cdot\left(\nu_0(h(\b{P})+1)+\nu_1\deg\b{P}\right)\Bigg),
\end{multline}
the first inequality in~\eqref{LdT_preuve_degX0etadelta} is a consequence of~\eqref{LdT_bound_degI0} and~\eqref{LdT_bound_hI0}, and the second follows by a direct application of definitions of $\eta$ and $\delta$.
The condition
$$
\Ord_{\ull{f}}\tilde{X}_0\geq c_n^{-1}\deg\left(\tilde{X}_0,\eta,\delta\right)
$$
is assured by direct comparison of the r.h.s in~\eqref{LdT_lb_ordI0} (recall that by definition $\Ord_{\ull{f}}\tilde{X}_0=\Ord_{\ull{f}}I_0$) and~\eqref{LdT_preuve_degX0etadelta}  and taking into account the hypothesis~\eqref{LdT_Cestgrande_2}.


The conclusion of Corollary~11 of~\cite{NP} gives us exactly the conclusion of the theorem. Indeed, this corollary provides us a cycle $\b{Z}\subset\tilde{X}_0(\ol{\kk(\b{z})})$ defined over $\kk(\b{z})$ and of dimension 0 such that
\begin{eqnarray}
  \delta\cdot h(\b{Z})+\eta\deg\b{Z} &\leq& \deg\left(\tilde{X}_0,(\eta,\delta)\right) \label{LdT_preuve_cor_conclusion_deg},\\
  \sum_{\ul{\alpha}\in\b{Z}} \Ord_{\ull{f}}(\ul{\alpha}) &>& c_n^{-1}\left(\eta\deg\b{Z} + \delta\cdot h(\b{Z})\right) \label{LdT_preuve_cor_conclusion_Ord}.
\end{eqnarray}
Inequality~(\ref{LdT_preuve_cor_conclusion_deg}) (together with~\eqref{LdT_preuve_degX0etadelta}) gives us inequality~(\ref{LdTdegZ}), and~(\ref{LdT_preuve_cor_conclusion_Ord}) provides us~(\ref{LdTordZ0}).
\end{proof}
\begin{definition} \label{defZP} Let $C$ be a real number satisfying~(\ref{LdT_Cestgrande}). We associate to each non-zero homogeneous polynomial $\b{P} \in \kk[\b{z}][X_1,...,X_n]$ and a real constant $C>0$ satisfying~(\ref{ordPplusqueb}) 
an irreducible 0-dimensional cycle $\b{Z}_{C}(\b{P})$ defined over $\kk(\b{z})$, contained in the zero locus of $P$ and satisfying
inequalities~(\ref{LdTdegZ}) and~(\ref{LdTordZ0}). In view of the transference lemma there exists at least one cycle satisfying all these conditions (provided polynomial $P$ and constant $C$ satisfy~(\ref{ordPplusqueb})). If there exists more than one such cycle, we choose one of them and fix this choice.
\end{definition}

\begin{remark} \label{genZP} Considering $(1:\b{z})$ as coordinates of a point
in $\mpp^1_{\overline{\kk(\b{z})}}$ we can consider the cycle $\b{Z}$ as an
$1$-dimensional cycle in $\mpp^1_{\kk}\times\mpp^n_{\kk}$ (defined over $\kk$).
In this case we denote this cycle by $\Z_C(P)$.

\smallskip

We associate to a bi-homogeneous polynomial $P(X_0',X_1',X_0,X_1,...,X_n)\in\AnneauDePolynomes$ satisfying
\begin{equation} \label{ordPplusqueb2}
\frac{\ordz (P(1,\b{z},\ull{f}) - (\deg_{\ul{X}}P) \ordz (\ull{f}) - \deg_{\ul{X}'}P}{t\cdot\left((\nu_0+\mu) \left(h(\b{P})+1\right)+(\nu_1+\mu)\deg P\right)\mu^{t-1}\left(\deg\b{P}+1\right)^{t}} > C,
\end{equation}
the homogeneous polynomial
$$
\tilde{P}(X_0,X_1,...,X_n)=P(1,\b{z},X_0,X_1,...,X_n)
$$
(satisfying in this case~(\ref{ordPplusqueb})). We have already defined the cycles $\b{Z}_C(\tilde{P})$ and $\Z_{C}(\tilde{P})$ for the latter polynomial, as $\tilde{P}\in\kk[z][X_0,\dots,X_n]$ (see Definition~\ref{defZP}). By this procedure we associate equally the cycles $\b{Z}_C(P)$ and $\Z_{C}(P)$ to every bi-homogeneous polynomial $P \in \AnneauDePolynomes$ satisfying~(\ref{ordPplusqueb2}).
\end{remark}
\begin{remark} \label{rem_def_delta}
Note that if $P\not\in\idpf$, that is if $P(\ul{f})\ne 0$, we necessarily have $\ul{f}\not\in Z_C(P)$ (because $Z_C(P)\in\Zeros(P)$ by definition), or in other terms $\I(Z_C(P))\setminus\idpf\ne\emptyset$. Hence quantities $\delta_0(Z_C(P))$ and $\delta_1(Z_C(P))$ (introduced in Definition~\ref{def_delta})  are defined if $P(\ul{f})\ne 0$.
\end{remark}
\begin{remark} \label{rem_LdT_Z_nonisotrivial} Note that combining~(\ref{ordPplusqueb2}) (for $C$ large enough) with the transference lemma (Theorem~\ref{LdT}, (\ref{LdTordZ0})) we can assure that the cycle $\Z_{C}(P)$ is not an isotrivial one (hence $\b{Z}_C(P)$ is not defined over~$\kk$). Indeed, 
each point defined over $\kk$ contributes at most $\Ordz\left(\ullt{f}\wedge\ull{f}(0)\right)$ to $\Ord_{\ull{f}}\Z_{C}(P)$, so for an isotrivial cycle $Z$ one has
$$
\Ord_{\ull{f}}Z\leq\Ordz\left(\ull{f}(\b{z})\wedge\ull{f}(0)\right)\deg Z.
$$
Thus
\begin{equation} \label{def_Ciso}
C>\Ciso:=\left(\frac{c_n\Ordz\left(\ull{f}\wedge\ull{f}(0)\right)+1}{\min(\nu_0,\mu)}\right)^n
\end{equation}
implies that $\Z_{C}(P)$ is not isotrivial.
\end{remark}

We recall the notation introduced in Definition~\ref{def_tf}. Let $f_1(z),\dots,f_n(z)$ be a set of functions, then we define
\begin{equation*}
 t=t(\ul{f}):=\trdeg_{\kk(z)}\kk(z,f_1(z),\dots,f_n(z)).
\end{equation*}

The following theorem plays an important role in the proof of our principal result, Theorem~\ref{LMGP}. 

\begin{theorem} \label{dist_alpha} Let $\ull{f}=(1:\b{f}_1:\dots:\b{f}_n) \in
\mpp^n_{\kk[[\b{z}]]}$ and let $\b{P}\in\kk[\b{z}][X_0,\dots,X_n]$ be a
homogeneous polynomial such that
\begin{equation*}
P(z,\b{f}_1(z),\dots,\b{f}_n(z))\ne 0.
\end{equation*}
Assume that $P$ satisfies~(\ref{ordPplusqueb}) with
\begin{equation} \label{dist_alpha_Cestgrande}
C \geq \max\left(\left(3t!c_n/\min(\nu_0,\mu)\right)^t,\left(\frac{c_nC_{sg}+1}{\min(\nu_0,\mu)}\right)^{t+1},\Ciso\right)
\end{equation}
(where constant $C_{sg}$ is described in Corollary~\ref{cor1_deltainfty} and $\Ciso$ in Remark~\ref{rem_LdT_Z_nonisotrivial}, \eqref{def_Ciso}).
Let $\b{Z}=\b{Z}_C(\b{P})$ and let $\b{P}_0 \in \kk[\b{z}][X_0,\dots,X_n]$ be a
homogeneous non-zero polynomial in $\ul{X}$, vanishing on $\b{Z}$ and realizing
the minimum of the quantity
\begin{equation} \label{dist_alpha_q}
 \nu_0\cdot h(\b{Z})h(\b{Q})+\nu_1\cdot\deg\b{Z}\cdot h(\b{Q})+\mu\cdot h(\b{Z})\deg_{\ul{X}}\b{Q}
\end{equation}
over all homogeneous polynomials $\b{Q}\in\kk[\b{z}][X_0,\dots,X_n]\setminus\idpf$
vanishing on $\b{Z}$. We denote $\delta_0:=h(\b{P}_0)$ and
$\delta_1:=\deg_{\ul{X}}\b{P}_0$ (cf. Definition~\ref{def_delta}).

There exists a point $\ull{\alpha}\in\b{Z}$ satisfying
\begin{equation} \label{Cdirect}
\Ord(\ull{f},\ull{\alpha})
> \tilde{C} (\delta_0+1)(\delta_1+1)^t,
\end{equation}
where  $\tilde{C}=C^{\frac{1}{t}}\min(\nu_0,\mu)\left(3\cdot t!\cdot c_n^{\frac{t}{t+1}}\right)^{-1}$.
\end{theorem}
\begin{proof}
We claim that there exists a point $\ull{\alpha}_1 \in \b{Z}_C(\b{P})$ satisfying
\begin{equation} \label{dist_alpha_Ord_assezgrand}
\Ord(\ull{\alpha}_1,\ull{f}) \geq C_{sg}.
\end{equation}
Indeed, by definition of $\b{Z}_C(\b{P})$ (see Definition~\ref{defZP}) we have for this cycle the lower bound~(\ref{LdTordZ0}). This inequality implies that there exists a point $\ull{\alpha}_1 \in \b{Z}_C(\b{P})$ satisfying $\Ord(\ull{\alpha}_1,\ull{f}) \geq c_n^{-1}\left[(c_nC)^{\frac{1}{t+1}}\min(\nu_0,\mu)\right]$ and we deduce~(\ref{dist_alpha_Ord_assezgrand}) from~(\ref{dist_alpha_Cestgrande}).

In view of Corollary~\ref{cor1_deltainfty} condition~\eqref{dist_alpha_Ord_assezgrand} provides us
\begin{equation} \label{dist_alpha_soit0_soit1}
\delta_0 \geq 2\cdot n!+1 \mbox{ or } \delta_1 \geq 4n.
\end{equation}

Let $(a,b)\in\mnn^2$. By linear algebra one can construct a bi-homogeneous
polynomial $Q_{(a,b)}=Q_{(a,b)}(X_0',X_1',X_0,X_1,...,X_n)\in\AnneauDePolynomes
\setminus \idpf$ of bi-degree $(a,b)$ and of vanishing order at
$\ullt{f}=\left(1,\b{z},1,f_1(z),\dots,f_n(z)\right)$ satisfying
\begin{equation}\label{ie_Qab}
    \ordz Q_{(a,b)}(\ullt{f}) \geq \lfloor\frac{1}{t!}(a+1)(b+1)^t\rfloor.
\end{equation}
Indeed, by definition of $t$ we can chose indexes $i_1,\dots,i_t$ in such a way that $z,f_{i_1},\dots,f_{i_t}$ are all algebraically independent over $\kk$. The space of all bi-homogeneous polynomials of bi-degree up to $(a,b)$ and depending only on variables $X_0',X_1',X_0,X_{i_1},\dots,X_{i_t}$ has dimension
$$
(a+1)\binom{b+t}{t}>\frac{1}{t!}(a+1)(b+1)^t
$$
over $\kk$, so we can choose among them a non-zero polynomial satisfying~\eqref{ie_Qab}. By construction this polynomial can not belong to $\idpf$, otherwise it would provide a non-trivial algebraic relation between $z,f_{i_1},\dots,f_{i_t}$ that is impossible by the choice of indexes $i_1,\dots,i_t$.

Let
\begin{equation} \label{dist_alpha_choix_ab}
(a,b)=\left\{\begin{aligned}&(\delta_0-1,\delta_1), \; \mbox{ if }\delta_0 \geq 2\cdot n!+1,\\
                     &(\delta_0,\delta_1-1), \; \mbox{ otherwise, i.e. } \delta_1 \geq 4n \mbox{ in view of~(\ref{dist_alpha_soit0_soit1}) }.\end{aligned}\right.
\end{equation}
We claim that for this choice of $(a,b)$ the following inequality holds
\begin{equation} \label{dist_alpha_est_ab}
\ordz Q_{(a,b)}(\ullt{f}) > \frac{1}{2\cdot t!}(\delta_0+1)(\delta_1+1)^t.
\end{equation}
In view of~(\ref{dist_alpha_choix_ab}), exactly two cases are possible:
\begin{list}{}{\usecounter{tmpabcdmine}}
\item a) \label{dist_alpha_cas_a} $\delta_0 \geq 2\cdot n!+1$,
\item b) \label{dist_alpha_cas_b} $\delta_1 \geq 4n$.
\end{list}
By~(\ref{ie_Qab}) we have
\begin{equation} \label{ordQ_gt_abm1}
\ordz Q_{(a,b)}(\ullt{f}) \geq
\lfloor\frac{1}{t!}(a+1)(b+1)^t\rfloor > \frac{1}{t!}(a+1)(b+1)^t -1.
\end{equation}
In the case~\ref{dist_alpha_cas_a} we proceed as follows. First, in this case~\eqref{ordQ_gt_abm1} can be rewritten as
\begin{equation}
\ordz Q_{(a,b)}(\ullt{f}) > \frac{1}{t!}\left(\delta_0(\delta_1+1)^t-t!\right)
\end{equation}
and in order to show~(\ref{dist_alpha_est_ab}) it is sufficient to verify
\begin{equation} \label{dist_alpha_suffit_cas_a}
2\delta_0(\delta_1+1)^t-2\cdot t!\geq (\delta_0+1)(\delta_1+1)^t.
\end{equation}
The latter inequality is obvious for $\delta_0 \geq 2\cdot n!+1\geq 2\cdot t!+1$ (and $\delta_1 \geq 0$).

In the case~\ref{dist_alpha_cas_b} the same procedure brings us to the point where it is sufficient to verify
(instead of~(\ref{dist_alpha_suffit_cas_a}))
\begin{equation} \label{dist_alpha_suffit_cas_b}
2(\delta_0+1)\delta_1^t-2\cdot t! \geq (\delta_0+1)(\delta_1+1)^t.
\end{equation}
We can rewrite this inequality as
\begin{equation} \label{ie_dist_alpha_suffit_cas_b}
\left(2\left(\frac{\delta_1}{\delta_1+1}\right)^t-1\right)(\delta_0+1) \geq \frac{2\cdot t!}{(\delta_1+1)^t}.
\end{equation}
 The l.h.s. of~(\ref{ie_dist_alpha_suffit_cas_b}) is an increasing function of $\delta_0$ and $\delta_1$, and the r.h.s. of~(\ref{ie_dist_alpha_suffit_cas_b}) is a decreasing function of $\delta_1$. So it is sufficient to verify this inequality for $\delta_0=0$ and $\delta_1=4t\leq 4n$. We can directly calculate
\begin{equation}
2\left(\frac{4t}{4t+1}\right)^t(0+1) > 1/2 > \frac{2\cdot t!}{(4t+1)^t},
\end{equation}
hence~(\ref{dist_alpha_suffit_cas_b}) is true for all the values $\delta_0 \geq 0$, $\delta_1 \geq 4n$. This completes the proof of~(\ref{dist_alpha_est_ab}).

We define
\begin{equation*}
\b{Q}(X_0,...,X_n)=Q_{(a,b)}(1,\b{z},X_0,...,X_n)^q,
\end{equation*}
where $q = \lceil 2\cdot t! \tilde{C} \rceil$; therefore we have
$\ordz Q_{(a,b)}(1,\b{z},1,\b{f}_1,...,\b{f}_n)^q \geq \tilde{C} (\delta_0+1)(\delta_1+1)^t$. As the polynomial $Q_{(a,b)}$ was constructed in a way to satisfy $Q_{(a,b)}(\ul{f})\ne 0$, we have $\b{Q}(f_1,\dots,f_n) \ne
0$, hence $Q\not\in\idpf$.

It is easy to verify
\begin{equation*}
\begin{aligned}
&h(\b{Q})\leq\deg_{\ul{X}'}Q_{(a,b)} = a,\\
&\deg_{\ul{X}}\b{Q}=\deg_{\ul{X}}Q_{(a,b)} = b.
\end{aligned}
\end{equation*}
We recall that in the statement we have introduced notation $Z=Z_C(P)$. One obviously has
\begin{equation*}
\deg\b{Z} \geq 1
\end{equation*}
and, as $\b{Z}$ is  not defined over $\kk$ (see Remark~\ref{rem_LdT_Z_nonisotrivial}), one has
\begin{equation*}
h(\b{Z}) \geq 1.
\end{equation*}
In view of~(\ref{dist_alpha_choix_ab}) we obtain that the polynomial $\b{Q}_{(a,b)}$ makes the quantity~(\ref{dist_alpha_q}) strictly smaller than the minimum realized by $\b{P}_0$. So $Q_{(a,b)}$ (and hence $Q$) can not vanish
on $\b{Z}$ (by the definition of $\b{P}_0$); in other words: $\b{Q}$ does not belong to $\I(\b{Z})$.

We apply Theorem~4.11 of chapter~3 of~\cite{NP} to
polynomial $\b{Q}(X_0,...,X_n)$
and to the ideal $\I(\b{Z})$, which is 0-dimensional over $\kk(\b{z})$.

Let $\ull{\alpha}\in\b{Z}$ realize the maximum of $\Ord(\cdot,\ull{f})$
for points of $\b{Z}$; in other words: let
$\Ord(\ull{f},\ull{\alpha})=\max_{\beta \in
\b{Z}}\Ord(\ull{f},\ull{\beta})$.

We define
\begin{equation} \label{def_theta}
\theta=\left\{ \begin{aligned} \Ordz \b{Q}(\ull{f}) &\mbox{, if }
\Ord(\ull{f},\ull{\alpha})>\Ordz \b{Q}(\ull{f})\\
\Ord_{\ull{f}} \I(\b{Z}) &\mbox{, if }
\Ord(\ull{f},\ull{\alpha}) \leq \Ordz \b{Q}(\ull{f})
\end{aligned}
\right.
\end{equation}

By Theorem~4.11 of chapter~3 of~\cite{NP} one has
\begin{equation} \label{thetaMaj}
\theta \leq h(\b{Q})\deg(\I(\b{Z}))+h(\I(\b{Z}))\deg(\b{Q})
\end{equation}
(in our case the base field is $\kk(\b{z})$ and all its valuations
are non-archimedean ones, so $\nu=0$ and the term $\nu m^2
\deg(\I(\b{Z})) \deg(\b{Q})$ is equal to zero in the statement of this theorem).

We claim that the inequality
\begin{equation} \label{dist_alpha_ie1}
\Ord(\ull{f},\ull{\alpha}) \leq \Ordz \b{Q}(\ull{f})
\end{equation}
is in fact impossible.

Indeed, in this case $\theta=\Ord_{\ull{f}} \I(\b{Z})$, so~(\ref{thetaMaj}) implies
\begin{equation*} 
\Ord_{\ull{f}} \I(\b{Z}) \leq q \delta_0 \deg(\b{Z}) + q \delta_1 h(\b{Z}),
\end{equation*}
and we can weaken this inequality
\begin{equation*} 
\Ord_{\ull{f}} \I(\b{Z}) \leq \frac{q}{\min(\nu_0,\mu)}\left(\nu_0\delta_0 \deg(\b{Z}) + \nu_1\delta_1\deg(\b{Z}) + \mu\delta_1 h(\b{Z})\right).
\end{equation*}
Using the definition of $\delta_0$ and $\delta_1$ we deduce
\begin{equation} \label{majorationZf}
\begin{aligned}
\Ord_{\ull{f}} \I(\b{Z}) &\leq \frac{q}{\min(\nu_0,\mu)}\Big(\nu_0 h(\b{P}_0) \deg(\b{Z})\\&\qquad\qquad\qquad\qquad + \nu_1\deg\b{P}_0\deg(\b{Z}) + \mu\deg\b{P}_0h(\b{Z})\Big)\\
&\leq \frac{q}{\min(\nu_0,\mu)}\Big(\nu_0 h(\b{P}) \deg(\b{Z})\\&\qquad\qquad\qquad\qquad + \nu_1\deg\b{P}\deg(\b{Z}) + \mu\deg\b{P}h(\b{Z})\Big).
\end{aligned}
\end{equation}
Further, as $\b{P}$ vanishes on $\b{Z}$ one has
\begin{multline*}
 \nu_0 h(\b{P}_0) \deg(\b{Z}) + \nu_1\deg\b{P}_0\deg(\b{Z}) + \mu\deg\b{P}_0h(\b{Z}) \\ \leq \nu_0 h(\b{P}) \deg(\b{Z}) + \nu_1\deg\b{P}\deg(\b{Z}) + \mu\deg\b{P}h(\b{Z})
\end{multline*}
by the minimality from the definition of  $\b{P}_0$.
Then, applying~(\ref{LdTordZ}) (recall our notation $\b{Z}=\b{Z}_C(\b{P})$), one has
\begin{equation*}
\Ord_{\ull{f}}\I(\b{Z})> C^{\frac{1}{t}}c_n^{-\frac{t}{t+1}} \left(\nu_0 h(\b{P}) \deg(\b{Z}) + \nu_1\deg\b{P}\deg(\b{Z}) + \mu\deg\b{P}h(\b{Z})\right)
\end{equation*}
and gluing this inequality with~(\ref{majorationZf}) we obtain
\begin{multline*}
 C^{\frac{1}{t}}c_n^{-\frac{t}{t+1}} \left(\nu_0 h(\b{P}) \deg(\b{Z}) + \nu_1\deg\b{P}\deg(\b{Z}) + \mu\deg\b{P}h(\b{Z})\right)\\ < \frac{q}{\min(\nu_0,\mu)}\left(\nu_0 h(\b{P}) \deg(\b{Z}) + \nu_1\deg\b{P}\deg(\b{Z}) + \mu\deg\b{P}h(\b{Z})\right).
\end{multline*}
Simplifying $\nu_0 h(\b{P}) \deg(\b{Z}) + \nu_1\deg\b{P}\deg(\b{Z}) + \mu\deg\b{P}h(\b{Z})$ we deduce inequality
\begin{equation*}
 3\cdot t!\tilde{C}=C^{\frac{1}{t}}\min(\nu_0,\mu)c_n^{-\frac{t}{t+1}} < q = \lceil 2\cdot t!\tilde{C} \rceil
\end{equation*}
which contradicts the definition of~$q$ and $\tilde{C} \geq 1$ (recall that $\tilde{C}$ is defined at the end of the statement of this theorem and $\tilde{C} \geq 1$ in view of~(\ref{dist_alpha_Cestgrande})). So the inequality~(\ref{dist_alpha_ie1}) is impossible.

Thus unavoidably one has
\begin{equation*}
\Ord(\ull{f},\ull{\alpha})>\ordz \b{Q}(\ull{f}).
\end{equation*}
By construction of $\b{Q}$ one has $\ordz \b{Q}(\ull{f}) > \tilde{C} (\delta_0+1)(\delta_1+1)^t$, so we
deduce
\begin{equation*}
\Ord(\ull{f},\ull{\alpha}) > \tilde{C} (\delta_0+1)(\delta_1+1)^t.
\end{equation*}
\end{proof}

\section{Principal result} \label{section_principal result}

In this section we introduce the main result of this paper 
and prove it.

Recall that general framework imposed for this article is given in subsection~\ref{subsection_general_framework}. So, we have an algebraically closed field $\kk$, a polynomial ring $\A=\kk[X_0',X_1',X_0,\dots,X_n]$ bi-graduated with respect to $\left(\deg_{\ul{X}'},\deg_{\ul{X}}\right)$, a point
\begin{equation*}
\ul{f}=\left(1:z,1:f_1(z):\dots:f_n(z)\right)
\end{equation*}
and a map $\phi:\A\rightarrow\A$ satisfying properties~(\ref{degphiQleqdegQ}) and~(\ref{condition_T2_facile}). We also recall the notation
\begin{equation} \label{def_r}
     t:=t_{\ul{f}}=\trdeg_{\kk(z)}\kk\left(f_1(z),\dots,f_n(z)\right).
\end{equation}
In the statement below as well as in the subsequent considerations we use various notions defined in subsections~\ref{definitions_comm_algebra} and~\ref{definitions_multiprojective_dg}. In particular, $m(I)$ (as well as $V_i$ and $e_{\phi}$) is defined in Definition~\ref{definDePP}, $i_0$ is defined in Definition~\ref{def_i0} and $\ord_{\ul{f}}$ is defined in Definition~\ref{defin_ord_xy}.
\begin{theorem}[Formal multiplicity lemma]\label{LMGP} Let $\kk$, $\A$, $\ul{f}$ and $\phi$ be as above. Assume that the map $\phi$ is $\ul{f}$-admissible.
Let $n_1\in\{1,\dots,n\}$ and $C_0, C_1\in\mrr^+$. We denote by $\cK_{n_1}$ the set of all equidimensional bi-homogeneous ideals $I\subset\AnneauDePolynomes$ of rank $\geq n_1$, such that $\idp_{\ul{f}}\subsetneq I$, $\ul{f}\not\in\V(I)$ and $m(I)\leq C_m$ (recall that the constant $C_m$ is introduced in Definition~\ref{def_Cm}),
and moreover such that all its associated prime ideals satisfy
\begin{equation} \label{theoLMGP_condition_ordp_geq_C0}
\ord_{\ull{f}}\idq \geq C_0.
\end{equation}
Assume also that $\ul{f}$ has the $\left(\phi,\cK_{n_1}\right)$-property (see Definition~\ref{def_weakDproperty}).

Then there exists a constant $K>0$ such that for all $P \in
\AnneauDePolynomes$, satisfying $P(1,z,1,f_1(z),\dots,f_n(z))\ne 0$ and for all $C\geq C_1$
\begin{equation}\label{condition_n1}
    i_0(\Z_C(P))\geq n_1
\end{equation}
(recall that the cycle $\Z_C(P)$ is introduced in Remark~\ref{genZP} and $i_0$ in the Definition~\ref{def_i0}),
satisfy also
\begin{equation} \label{LdMpolynome2}
\ordz(P(\ullt{f})) \leq K\left((\mu+\nu_0)(\deg_{\ul{X}'}P+1)+\nu_1\deg_{\ul{X}}P\right)\\ \times\mu^{n-1}(\deg_{\ul{X}} P + 1)^t.
\end{equation}
\end{theorem}

\begin{remark} \label{rem_n1}
Condition~(\ref{condition_n1}) is tautologically true with parameters $n_1=1+r_{\ul{f}}$ and $C_1=0$ (in view of the definition of $i_0(Z(P))$, see Definition~\ref{def_i0} and Remark~\ref{rem_i0}).
Using this choice of $n_1$ and $C_1$ and 
enlarging class $\cK_{n_1}$ to $\cK_{1+r_{\ul{f}}}$ we obtain the statement of Theorem~\ref{intro_LMGP}.
\end{remark}
\begin{remark}
Parameters $n_1$ and $C_1$ are introduced because in certain situations it is possible to provide direct lower estimate of $i_0(Z(P))$ better than 1 (see Remark~\ref{rem_i0}), so excluding the necessity of analysis of $\phi$-stable ideals of a small codimension. Sometimes it could appear a decisive step, e.g. see proof of Proposition~4.11 in~\cite{EZ2010}.
\end{remark}



We deduce Theorem~\ref{LMGP} at the end of this section as a consequence of Lemma~\ref{LemmeProp13} and the following Proposition~\ref{PropositionLdMprincipal}.

\begin{proposition} \label{PropositionLdMprincipal}
Let $P \in \AnneauDePolynomes$
and $C\in\mrr$ satisfy
$$
P(1,z,1,f_1(z),\dots,f_n(z))\ne 0
$$
and:
\begin{align}\label{C_bornee_enP}
C &< \frac{\ordz (P \circ \ullt{f}) - (\deg_{\ul{X}}P) \ordz (\ull{f})-\deg_{\ul{X}'}P}{t\left((\nu_0+\mu)(\deg_{\ul{X}'}P+1)+\nu_1\deg_{\ul{X}}P\right)\mu^{t-1}(\deg_{\ul{X}}P+1)^{t}}\\
C &\geq \left(\min(\nu_0,\mu)\right)^{-t}. \label{C_minore_numu}
\end{align}
Let $\idp$ be the ideal defined as $\idp=\I(\Z_C(P))$, where $\Z_C(P)$ is the cycle introduced in Remark~\ref{genZP}.
Assume that for $i=i_0(\Z_C(P))$ one has
\begin{equation} \label{majoration_e_par_m}
e_{\phi}(V_i(\idp),\idp) \leq m(I(V_i(\idp),\idp)).
\end{equation}
Then,
\begin{equation}\label{Cestsmall}
    C \leq \frac{(2\rho_{n+1})^tc_n^{t}}{\min(1;\lambda)^t\min(1;\mu)^t}.
\end{equation}
Moreover, for all polynomials $P\in\AnneauDePolynomes$, one has
\begin{equation} \label{estimationP}
\begin{aligned}
 \ordz &P(\ullt{f}(\b{z}))\leq\max\left(\frac{t}{\left(\min(\nu_0,\mu)\right)^{t}},\frac{(2\rho_{n+1})^tc_n^{t}}{\min(1;\lambda)^t\min(1;\mu)^t}\right)\\
 &\times\left((\mu+\nu_0)(\deg_{\ul{X}'}P+1)+\nu_1\deg_{\ul{X}}P\right)\mu^{t-1}(\deg_{\ul{X}} P+1)^t\\
 &+(\ordz \ull{f})(\deg_{\ul{X}} P)+\deg_{\ul{X}'}P.
\end{aligned}
\end{equation}
\end{proposition}
\begin{proof}
Note that if $\deg_{\ul{X}}P=0$, then the conclusion of the proposition 
is automatically satisfied. Thus we need only to treat the case $\deg_{\ul{X}}P\geq 1$.

{\it Ad absurdum} assume
\begin{equation} \label{cnstCgrande}
 C > \frac{(2\rho_{t+1})^tc_n^{t}}{\min(1;\lambda)^t\min(1;\mu)^t}.
\end{equation}

Recall that $i_0=i_0(\Z_C(P))\geq 1$ is
the largest index $i \in \{1,...,n\}$ such that $\rg(V_i\A_{\idp}) \geq i+r_{\ul{f}}$ (see Definition~\ref{def_i0}).
We put $e_0$ the largest integer $\leq e_{\phi}(V_{i_0},\idp)$
such that $V_{i_0}+...+{\phi}^{e_0}(V_{i_0})\subset\idp$ (we use the notation $V_{i_0}$ as a shorthand for $V_{i_0}(\idp)$). Note that the assumption~(\ref{majoration_e_par_m}) implies that
$e_{\phi}(V_{i_0},\idp)$ is finite, so $e_0$ is a well-defined integer.

\smallskip

Let $Q$ be a generator of ${\phi}^{e_0}(V_{i_0})$; by Lemma~\ref{majorationphinQ} one has
\begin{equation} \label{PropositionLdMprincipal_majoration_deg_Q}
\begin{aligned}
&\deg_{\ul{X}}Q \leq \mu^{e_0}\rho_{i_0}\delta_1(\idp),\\
&\deg_{\ul{X}'}Q \leq (\nu_0\delta_0(\idp)+e_0\nu_1\delta_1(\idp))\max(\nu_0,\mu)^{e_0-1}\rho_{i_0}.
\end{aligned}
\end{equation}

With the substitution $(X_0':X_1')=(1:\b{z})$ we can consider $Q$ as a polynomial from $\kk[\b{z}][X_0:...:X_n]$.

We denote $\b{Z}=\b{Z}_C(P)$. Let $\b{\alpha}\in\b{Z}$. By Lemma~\ref{Representants}, b), there is a system of projective coordinates $\ull{\alpha}$ satisfying
\begin{eqnarray*}
 \ordz\ull{\alpha} &=& \ordz\ull{f}, \\
  \ordz(\ull{\alpha} - \ull{f}) - \ordz(\ull{f}) &=& \Ordz(\b{\alpha},\b{f}).
\end{eqnarray*}
In view of $\ordz\ull{f}=0$, we deduce immediately
\begin{eqnarray}
  \ordz\ull{\alpha} &=& 0,\label{PropositionLdMprincipal_alpha_c1} \\
  \ordz(\ull{\alpha} - \ull{f}) &=& \ordz(\ull{\alpha}\wedge\ull{f}). \label{PropositionLdMprincipal_alpha_c2}
\end{eqnarray}

We fix a choice of projective coordinate systems satisfying~(\ref{PropositionLdMprincipal_alpha_c1}) and~(\ref{PropositionLdMprincipal_alpha_c2})
for all $\b{\alpha}\in\b{Z}$.

We claim that for any $\alpha\in Z$
\begin{equation} \label{point_Crucial}
 \begin{aligned}
  \ordz(\b{\phi}(\b{Q})(\ull{\alpha}))
	&\geq \min(\ordz(\b{\phi}(\b{Q})(\ull{f})),\ordz(\ull{\alpha}\wedge\ull{f})).
 \end{aligned}
\end{equation}
Indeed,
\begin{equation*}
\begin{aligned}
    \ordz&\left(\b{\phi}(\b{Q})(\ull{\alpha})\right)=\ordz\left(\left(\b{\phi}(\b{Q})(\ull{\alpha})-\b{\phi}(\b{Q})(\ull{f})\right)+\b{\phi}(\b{Q})(\ull{f})\right)\\
    &\geq\min\left(\ordz\left(\b{\phi}(\b{Q})(\ull{\alpha})-\b{\phi}(\b{Q})(\ull{f})\right),\ordz\left(\b{\phi}(\b{Q})(\ull{f})\right)\right)\\
    &\geq\min\left(\ordz\left(\ull{\alpha}-\ull{f}\right),\ordz\left(\b{\phi}(\b{Q})(\ull{f})\right)\right)\\
    &\geq\min\left(\ordz\left(\ull{\alpha}\wedge\ull{f}\right),\ordz\left(\b{\phi}(\b{Q})(\ull{f})\right)\right).
\end{aligned}
\end{equation*}
Then, by~(\ref{condition_T2_facile}) and in view of $\b{Q}(\ull{\alpha})=0$ (according to the choice of $e_0$),
\begin{equation} \label{point_Crucial_1}
 \begin{split}
  \ordz\left(\b{\phi}(\b{Q})(\ull{f})\right) &\geq \lambda \ordz\b{Q}(\ull{f})\\
	&\geq \lambda \ordz\big(\b{Q}(\ull{f})-\b{Q}(\ull{\alpha})\big)\\
	&\geq \lambda\ordz\left(\ull{\alpha}\wedge\ull{f}\right).
 \end{split}
\end{equation}
We deduce from~(\ref{point_Crucial}) and~(\ref{point_Crucial_1})
\begin{equation} \label{PropositionLdMprincipal_ie0}
 \ordz(\b{\phi}(\b{Q})(\ull{\alpha})) \geq \min(1,\lambda) \ordz(\ull{\alpha}\wedge\ull{f})
\end{equation}
for all $\b{\alpha}\in\b{Z}$.

By~(\ref{PropositionLdMprincipal_ie0}) one has
\begin{equation} \label{PropositionLdMprincipal_ie1_1}
 \begin{split}
  &\frac{1}{\min(1,\lambda)}\sum_{\ull{\alpha} \in \b{Z}_{C}(P)}\left(\ordz\left(\phi(\b{Q})(\ull{\alpha})\right)\right)
  \\&\geq\sum_{\ull{\alpha} \in \b{Z}_{C}(P)}\left(\ordz(\ull{\alpha}\wedge\ull{f})\right)=:M
 \end{split}
\end{equation}
(note that $M$ is equal to the l.h.s. of~(\ref{LdTordZ})). By definition of $\b{Z}_C(P)$ (see Definition~\ref{defZP} and Remark~\ref{genZP}) and with~(\ref{cnstCgrande}) we estimate
\begin{multline} \label{PropositionLdMprincipal_ie1}
  M>C^{\frac{1}{t}}c_t^{-\frac{t}{t+1}}\left(\nu_0\deg(\b{Z})\deg_{\b{z}}P+\nu_1\deg(\b{Z})\deg_{\ul{X}}P + \mu h(\b{Z})\deg_{\ul{X}}P\right)\\
	\geq \frac{2\rho_{n+1}}{\min(1,\lambda)\min(1,\mu)}\Big(\nu_0\deg(\b{Z})\deg_{\b{z}}P\\ + \nu_1\deg(\b{Z})\deg_{\ul{X}}P + \mu h(\b{Z})\deg_{\ul{X}}P\Big).
\end{multline}
We deduce from~(\ref{PropositionLdMprincipal_ie1_1}) and (\ref{PropositionLdMprincipal_ie1}) 
\begin{multline} \label{PropositionLdMprincipal_ie15}
  \sum_{\ull{\beta} \in \b{Z}_{C}(P)}\ordz\left(\phi(\b{Q})(\ul{\beta})\right)
	>\frac{2\rho_{n+1}}{\min(1,\mu)}\\ \times\left(\nu_0\deg(\b{Z})\deg_{\b{z}}P + \nu_1\deg(\b{Z})\deg_{\ul{X}}P + \mu h(\b{Z})\deg_{\ul{X}}P\right).
\end{multline}
Also Liouville's inequality~(\ref{iet_main})
implies (under assumption that $\b{\phi}(\b{Q})$ does not vanish on $\b{Z}_{C}(P)$)
\begin{equation} \label{PropositionLdMprincipal_ie2}
 \begin{split}
  &\sum_{\ull{\beta} \in \b{Z}_{C}(P)}\ordz\left(\phi(\b{Q})(\ul{\beta})\right)\leq\deg(\b{Z})h(\phi(\b{Q})) + h(\b{Z})\deg\phi(\b{Q})\\
  &\leq \max(\mu,\nu_0)^{e_0}\rho_{i_0}\left(\nu_0\deg(\b{Z})\delta_0 + \nu_1(e_0+1)\deg(\b{Z})\delta_1 + \mu h(\b{Z})\delta_1\right)\\
  &\leq \max(\mu,\nu_0)^{e_0}(e_0+1)\rho_{i_0}\left(\nu_0\deg(\b{Z})\delta_0 + \nu_1\deg(\b{Z})\delta_1 + \mu h(\b{Z})\delta_1\right)
 \end{split}
\end{equation}
(the second inequality in~\eqref{PropositionLdMprincipal_ie2} is implied by~(\ref{PropositionLdMprincipal_majoration_deg_Q})).

According to the definition of $e_0$, the hypothesis~(\ref{majoration_e_par_m}) and Lemma~\ref{LemmeCor14NumberW} we have
\begin{equation} \label{PropLdMprincipal_ie_e0_leq_m}
e_0 \leq e_{\phi}(V_{i_0},\idp) \leq m(I(V_{i_0},\idp)) \leq \nu(n+1)!\rho_{i_0}^{n+1},
\end{equation}
hence $\max(\mu,\nu_0)^{e_0}(e_0+1)\rho_{i_0}\leq \rho_{i_0+1} \leq \rho_{n+1}$ by definition of $\rho_{n+1}$.
Thus~(\ref{PropositionLdMprincipal_ie15}) and~(\ref{PropositionLdMprincipal_ie2}) lead to:
\begin{multline*}
    \frac{2\rho_{n+1}}{\min(1,\mu)}\left(\nu_0\deg(\b{Z})\deg_{\b{z}}P + \nu_1\deg(\b{Z})\deg_{\ul{X}}P + \mu h(\b{Z})\deg_{\ul{X}}P\right)\\
    <\rho_{n+1}\left(\nu_0\deg(\b{Z})\delta_0 + \nu_1\deg(\b{Z})\delta_1 + \mu h(\b{Z})\delta_1\right).
\end{multline*}
This inequality contradicts Definition~\ref{def_delta}, thus $\b{\phi}(\b{Q})(\ull{\alpha})=0$.

So we have
\begin{equation} \label{PropositionLdMprincipal_incl_e0p1}
{\phi}^{e_0+1}(V_{i_0}) \subset \idp,
\end{equation}
and this inclusion contradicts the definition of $e_0$ if $e_0 < e_{\phi}(V_{i_0},\idp)$. We conclude
$e_0=e_{\phi}(V_{i_0},\idp)$.

Moreover, (\ref{PropositionLdMprincipal_incl_e0p1}) implies
\begin{equation} \label{PropositionLdMprincipal_est_rg_rgP}
\rg\left((V_{i_0}+...+\phi^{e_0+1}(V_{i_0}))\AnneauDePolynomes_{\idp}\right)\leq\rg\left(\idp\AnneauDePolynomes_{\idp}\right)=n.
\end{equation}
As $e_0+1>e_{\phi}(V_{i_0},\idp)$ and by definition of $e_{\phi}(V_{i_0},\idp)$ we have
\begin{equation} \label{e0plus1Total}
\begin{split}
 \rg\left((V_{i_0}+...+\phi^{e_0+1}(V_{i_0}))\AnneauDePolynomes_{\idp}\right) > \rg(V_{i_0}\AnneauDePolynomes_\idp) \geq i_0,
\end{split}
\end{equation}
we obtain
\begin{equation*}
 \rg(V_{i_0+1}\AnneauDePolynomes_\idp)\geq\rg\left((V_{i_0}+...+\phi^{e_0+1}(V_{i_0}))\AnneauDePolynomes_{\idp}\right)\geq i_0+1.
\end{equation*}
If $i_0<n$ this inequality contradicts the definition of $i_0$ (Definition~\ref{def_i0}), and if $i_0=n$ inequality~(\ref{e0plus1Total})
implies
\begin{equation*}
\rg\left((V_{i_0}+...+\phi^{e_0+1}(V_{i_0}))\AnneauDePolynomes_{\idp}\right) > n,
\end{equation*}
in contradiction with~(\ref{PropositionLdMprincipal_est_rg_rgP}).

So, we have verified that the hypothesis~(\ref{cnstCgrande}) can not be satisfied, establishing therefore~(\ref{Cestsmall}).

It remains to verify~(\ref{estimationP}). We fix an arbitrary polynomial $P\in\A$ and consider the set $\mathcal{M}(P)$ of reals $C$ satisfying (with our choice of polynomial $P$) inequalities~(\ref{C_bornee_enP}) and~(\ref{C_minore_numu}).
If $\mathcal{M}(P)=\emptyset$, we have
\begin{equation*}
    \left(\min(\nu_0,\mu)\right)^{-n}\geq\frac{\ordz (P \circ \ullt{f}) - (\deg_{\ul{X}}P) \ordz (\ull{f})-\deg_{\ul{X}'}P}{n\left((\mu+\nu_0)(\deg_{\ul{X}'}P+1)+\nu_1\deg_{\ul{X}}P\right)\mu^{n-1}(\deg_{\ul{X}}P+1)^{n}}
\end{equation*}
obtaining immediately~(\ref{estimationP}).

In the opposite case, if $\mathcal{M}(P)\ne\emptyset$, we let $C_s$ denote the upper bound of $\mathcal{M}(P)$, which is a real finite number: in fact the inequality~(\ref{C_bornee_enP}) implies
\begin{equation*}
    C_s=\frac{\ordz (P \circ \ullt{f}) - (\deg_{\ul{X}}P) \ordz (\ull{f})-\deg_{\ul{X}'}P}{n\left((\mu+\nu_0)(\deg_{\ul{X}'}P+1)+\nu_1\deg_{\ul{X}}P\right)\mu^{n-1}(\deg_{\ul{X}}P+1)^{n}}.
\end{equation*}
In the first part of the proof we have established the inequality~(\ref{Cestsmall}) for all the elements of $\mathcal{M}(P)$, therefore $C_s$ also satisfies this
inequality, hence~(\ref{estimationP}).
\end{proof}


Now we are ready to prove the main result of this paper,  Theorem~\ref{LMGP}.

\begin{proof}[Proof of Theorem~\ref{LMGP}] We define
\begin{equation} \label{def_Ktwo}
K_2:=(n-r_{\ul{f}})\left(\deg_{(0,\dim\idp_{\ul{f}})}\idp_{\ul{f}}+\deg_{(1,\dim\idp_{\ul{f}}-1)}\idp_{\ul{f}}\right)
\\
\left(1+\frac{\nu_1}{\max(\mu,\nu_0)}\right)\rho_n^{n}
\end{equation}
and
\begin{equation} \label{Cestgrande}
 C=1+\max\Bigg(c_n^{t}C_0^{t}, \left(\min(\nu_0,\mu)\right)^{t}, \Ciso,\left(\frac{3n!c_n}{\min(\nu_0,\mu)}K_0K_2\right)^{t},C_1\Bigg)
\end{equation}
(recall that $c_n$ is defined in~(\ref{def_cn}), $C_0$, $C_1$ are introduced in the statement of Theorem~\ref{LMGP}) and the constant $K_0$ is implied by the $\left(\phi,\cK_{n_1}\right)$-property (this property is imposed in the statement of Theorem~\ref{LMGP} as well).

Let $P \in \AnneauDePolynomes\setminus\idpf$ be a polynomial that does not satisfy~(\ref{LdMpolynome2}) for
\begin{equation}\label{preuve_theo_defK}
    K=\max\left(2nC,\left(\frac{2\rho_{n+1}c_n}{\max(1,\lambda)\max(1,\mu)}\right)^t\right).
\end{equation}
Then it satisfies~(\ref{ordPplusqueb2}). In particular, $C$ and $P$ satisfy
hypothesis~(\ref{C_bornee_enP}) and~(\ref{C_minore_numu}) of Proposition~\ref{PropositionLdMprincipal}.

Define $\idp:=\I(\Z_C(P))$, where $\Z_C(P)$ is the cycle introduced in Remark~\ref{genZP}.
In view of~(\ref{LdTordZ}) and~(\ref{Cestgrande}), we have $\ord_{\ullt{f}}\idp > C_0$ and thus 
$\phi$ is correct with respect to $\idp$. Moreover, $\Z_C(P)$ is projected onto $\mpp^1$ (see Remark~\ref{rem_LdT_Z_nonisotrivial}).

Recall that $V_i=V_i(\Z_C(P))$ is introduced in Definition~\ref{V_irho_i} and
$e_{\phi}$, $m$ are introduced
in Definition~\ref{definDePP}: (\ref{definDePP_defin_e}) and (\ref{definDePP_defin_m}) respectively.
If for $i=i_0(\Z_C(P))$ we have
\begin{equation}\label{e_bornee}
    e_{\phi}(V_i,\idp) \leq m(I(V_{i},\idp)),
\end{equation}
we verify~(\ref{majoration_e_par_m}) and we can apply Proposition~\ref{PropositionLdMprincipal}.
 This proposition gives us~(\ref{LdMpolynome2}) in view of our choice of~$K$~\eqref{preuve_theo_defK}. This estimate contradicts our hypothesis that $P$ does not satisfy~(\ref{LdMpolynome2}).
On the other hand, if~(\ref{e_bornee}) is not satisfied, we apply Lemma~\ref{LemmeProp13} to the ideal $\idp$
and the $\kk$-linear space $V=V_i(\idp)$ (we recall the notation $i=i_0(\Z_C(P))$).

We denote by $J$ the equidimensional $\phi$-stable ideal provided by
Lemma~\ref{LemmeProp13}. In view of the property~$b)$ of this proposition we have
\begin{equation}\label{prove_LMGP_rkJ_eq_rkV_i}
    \rg(J)=\rg\left(I\left(V_i,\idp\right)\right)\geq i\geq n_1.
\end{equation}
The property~a) ensures $\ul{f}\not\in\V(J)$, because the ideal $I(V_{i},\idp)$ contains at least one polynomial that does not vanish at $\ul{f}$.

We verify (In view of~(\ref{LemmeProp13_en_part}))
\begin{equation}
\begin{aligned}
& m(J) \leq m(I(V,\idp)),\\
& \dd_{(0,n-\rg J+1)}(J) \leq \dd_{(0,n-\rg J+1)}(I(V_i,\idp)).
\end{aligned}
\end{equation}
As $\V(\idp)=Z_C(P)$ is projected onto $\mpp^1$ we have by Lemma~\ref{LemmeCor14NumberW}
$$
m(I(V,\idp)) \leq C_m
$$
and also $\delta_1(\idp) \geq 1$.

Recall that $I(V_i,\idp)\subset\idp$ and thus
\begin{equation*}
r:=\rg I(V_i,\idp) \leq \rg\idp = n.
\end{equation*}
As the ideal~$I(V_i,\idp) \subset \idp$ is extended-contracted of an ideal generated by polynomials of bi-degree
$\leq \left(\rho_i\left(\delta_0(\idp)+\frac{\nu_1}{\max(\mu,\nu_0)}\delta_1(\idp)\right), \rho_i\delta_1(\idp)\right)$ (see Definitions~\ref{def_I} and~\ref{V_irho_i}), we have by Lemma~\ref{lemma_BT}
\begin{multline} \label{est_basique1}
\dd_{(0, n-\rg I(V_i,\idp)+1)}(I(V_i,\idp))
\\
\leq(r-r_{\ul{f}})\left(\deg_{(1,\dim\idp_{\ul{f}}-1)}\idp_{\ul{f}}\right)\left(\delta_0(\idp)+\frac{\nu_1}{\max(\mu,\nu_0)}\delta_1(\idp)\right)
\\
\times
\delta_1(\idp)^{r-r_{\ul{f}}-1}\rho_i^{r-r_{\ul{f}}}
+\left(\deg_{(0,\dim\idp_{\ul{f}})}\idp_{\ul{f}}\right)\delta_1(\idp)^{r-r_{\ul{f}}}\rho_i^{r-r_{\ul{f}}}
\end{multline}
where we use the notation $r:=\rg I(V_i,\idp)$. 
Let's temporarily denote by $R(\delta_0,\delta_1)$ the r.h.s. of~\eqref{est_basique1}.
Using~\eqref{prove_LMGP_rkJ_eq_rkV_i} we infer
\begin{equation} \label{estdegW}
 \dd_{(0, n-\rg J+1)}J \leq R(\delta_0,\delta_1).
\end{equation}
As $J$ is an equidimensional ideal, we obtain for all $\idq\in\Ass(\AnneauDePolynomes/J)$
\begin{equation} \label{estdegQ}
\begin{aligned}
 \dd_{(0, n-\rg\idq+1)}\idq&\leq\dd_{(0, n-\rg J+1)}J\\
 &\leq R(\delta_0,\delta_1).
\end{aligned}
\end{equation}
The same calculation for $\dd_{(1, n-\rg\idq)}\idq$ gives us
\begin{equation} \label{estdeg1Q}
\begin{aligned}
 \dd_{(1, n-\rg\idq)}\idq
 \leq\rho_n^{n-r_{\ul{f}}}\left(\dd_{(1, n-\rg\idp_{\ul{f}})}\idp_{\ul{f}}\right)\left(\delta_1(\idp)+1\right)^{r-r_{\ul{f}}}.
\end{aligned}
\end{equation}
Summing up~\eqref{estdegQ} and~\eqref{estdeg1Q} we find, for every $\idq\in\Ass(\AnneauDePolynomes/J)$,
\begin{multline}
 \dd_{(0, n-\rg\idq+1)}\idq+\dd_{(1, n-\rg\idq)}\idq
\\
\leq R(\delta_0,\delta_1)+\rho_n^{n-r_{\ul{f}}}\left(\dd_{(1, n-\rg\idp_{\ul{f}})}\idp_{\ul{f}}\right)\left(\delta_1(\idp)+1\right)^{r-r_{\ul{f}}}
\\
\leq (r-r_{\ul{f}})\left(\deg_{(1,\dim\idp_{\ul{f}}-1)}\idp_{\ul{f}}\right)
\\
\times
\left(\delta_0(\idp)+\frac{\nu_1}{\max(\mu,\nu_0)}\delta_1(\idp)\right)
\delta_1(\idp)^{r-r_{\ul{f}}-1}\rho_i^{r-r_{\ul{f}}}
\\
+\left(\deg_{(0,\dim\idp_{\ul{f}})}\idp_{\ul{f}}+\deg_{(1,\dim\idp_{\ul{f}}-1)}\idp_{\ul{f}}\right)\delta_1(\idp)^{r-r_{\ul{f}}}\rho_i^{r-r_{\ul{f}}}
\\
\leq K_2\left(\delta_0(\idp)+1\right)\left(\delta_1(\idp)+1\right)^{n-r_{\ul{f}}},
\end{multline}
where $K_2$ is defined in~\eqref{def_Ktwo}. 

As $P$ and $C$ satisfy~(\ref{ordPplusqueb2}), by Lemma~\ref{dist_alpha} there exists a point $\ull{\alpha}\in\b{Z}_C(P)$
satisfying~(\ref{Cdirect}) with $\tilde{C} = \frac{C^{\frac{1}{t}}\min(\nu_0,\mu)}{3t!c_n}\geq K_0K_2$
(the last inequality is implied by the definition~(\ref{Cestgrande})),
and thus one has for all $\idq\in\Ass(\AnneauDePolynomes/J)$ (in view of Lemma~\ref{LemmeProp13}, point~\ref{LemmeProp13_c})
\begin{equation} \label{estordQ}
 \ord_{\ull{f}}\idq \geq \ord(\ull{f},\ull{\alpha}) > K_0K_2(\delta_0(\idp)+1)(\delta_1(\idp)+1)^{t_{\ul{f}}}.
\end{equation}
We recall that $r_{\ul{f}}$ and $t_{\ul{f}}$ are introduced in Definition~\ref{def_tf} and satisfy
$$
t_{\ul{f}}+r_{\ul{f}}=n.
$$

The estimates~(\ref{estdegQ}), (\ref{estdeg1Q})
and~(\ref{estordQ}) put together (and been verified for \emph{all} $\idq\in\Ass(\AnneauDePolynomes/J)$)
contradict the $\left(\phi,\cK\right)$-property assumed in the statement. So, the assumption~(\ref{ordPplusqueb2})
with $C$ given by~(\ref{Cestgrande}) is untenable and we deduce~(\ref{LdMpolynome2}) with our choice of~$K$.
It contradicts again our assumption that $P$ does not satisfy~(\ref{LdMpolynome2}).

Finally, we conclude that polynomial~$P$, which does not satisfy~(\ref{LdMpolynome2}), can not exist, and this completes the proof.
\end{proof}

\section{Applications}

Our Theorem~\ref{LMGP} extends the corresponding result from~\cite{EZ2010,EZ2011} to sets of functions $f_1(z),\dots,f_n(z)\in\kk[[z]]$ that possibly admit algebraic relations over $\kk(z)$. Once proving this generalization, we can immediately apply it to the cases of Mahler's functions and to the case of solutions of a differential system (with polynomial relations). In all these cases proofs can be transferred word for word from the case of algebraically independent functions, just replacing the references to formal multiplicity lemma with references to our Theorem~\ref{LMGP}.

In this section we only give statements of theorems and recall the corresponding frameworks. For proofs, we give references to the corresponding proofs in~\cite{EZ2010} and~\cite{EZ2011}.

\subsubsection*{Zero order estimates for functions satisfying functional equations of generalized Mahler's type}


Let $A_0,...,A_n$ be polynomials with coefficients in $\kk$ and satisfying $\deg_z A_i\leq s$, $\deg_{\ul{X}}A_i\leq q$. Let $p(z)$ be a rational  fraction with $\delta=\ord_{z=0} p({z})$  and $d:=\deg p(z)$.

We consider the following system of functional equations
\begin{equation} \label{relsTopfer}
{f_i}({p(z)}) =
\frac{A_i(z, f_1(z),...,f_n(z))}{A_0(z, f_1(z),...,f_n(z))},\quad i=1,\dots,n.
\end{equation}
Let $\T$ be a rational  map from $\mpp^1 \times \mpp^{n}$ to itself
defined by
\begin{multline} \label{defT2}
(X_0':X_1',X_0:...:X_n)\rightarrow\Big(A_0'(X_0',X_1'):A_1'(X_0',X_1'),\\
A_0(X_0',X_1',X_0,...,X_n):...:A_n(X_0',X_1',X_0,...,X_n)\Big),
\end{multline}
where $A_i'\in\kk[X_0',X_1']$, $i=0,1$, are homogeneous polynomials of degree $r$ in $\ul{X}'$ and $A_j\in\AnneauDePolynomes$, $j=0,...,n$,
are bi-homogeneous polynomials of bi-degree $(s,q)$ in $\ul{X}'$ and $\ul{X}$.

\begin{remark}\label{rem_Mutual_Association} We define $p(\b{z})=\frac{A_1'(1,\b{z})}{A_0'(1,\b{z})}$ and associate to every rational map $\T$ defined by~(\ref{defT2}) and such that $A_0$, $A_0'$
are non-zero polynomials,
a system of functional equations~(\ref{relsTopfer}):
\begin{equation} \label{relsTopfer2}
    A_0(\ullt{f}(\b{z}))f_i(p(\b{z}))=A_i(\ullt{f}(\b{z})), \quad i=1,...,n
\end{equation}
(where $\ullt{f}$ denotes $(1,\b{z},1,f_1(\b{z}),...,f_n(\b{z}))$). 

The other way around, is to start from the system~(\ref{relsTopfer}), then
formulae~(\ref{defT2}) define a morphism $\T:\mpp^1 \times \mpp^{n}\rightarrow\mpp^1 \times \mpp^{n}$.
\end{remark}

\begin{definition} \label{def_Mutual_Association}
We say that the morphism defined by~(\ref{defT2}) and the system~(\ref{relsTopfer2}) are mutually associated.
\end{definition}

\begin{definition}\label{def_irrT}
For any morphism $\T$ defined by~(\ref{defT2}), we denote by $\irrT$ the union of zero loci of the polynomial bi-homogeneous systems $A_i'(X_0',X_1',X_0,...,X_n)=0$, $i=0,1$, and $A_j(X_0',X_1',X_0,...,X_n)=0$, $j=0,...,n$. One has $\irrT \subset \mpp_{\kk}^1\times\mpp_{\kk}^n$ and this is a set of points where bi-projective application $\T$ is not well-defined (if $\irrT=\emptyset$ the map $\T$ is a regular bi-projective map).
\end{definition}

\begin{remark} \label{rem_TV_simplified}
To simplify the notation we write $\T(W)$ instead of $\T(W\setminus\irrT)$.
\end{remark}

\begin{definition}\label{definVarieteTstable}
We say that a sub-variety $W\subset\mpp^1\times\mpp^n$
is $\T$-stable, if
\begin{equation*}
 \ol{\T(W)}=W.
\end{equation*}
\end{definition}

\begin{remark} \label{varietestable_et_idealstable}
If a variety $W$ is $\T$-stable then the ideal $\I(W)$ is $\Talg$-stable, but the reciprocal statement is not true. The condition $\Talg(\I(W))\subset\I(W)$ geometrically means only that $W$ is a sub-scheme of $\T^{-1}(W)$. However, if we impose that the variety $W$ is irreducible and $\dim \T(W)=\dim W$, then
\begin{equation}
 \mbox{a variety $W$ is $\T$-stable} \Leftrightarrow \mbox{the ideal $\I(W)$ is $\Talg$-stable}.
\end{equation}
\end{remark}




\begin{theorem}\label{LMGPF} Let $\kk$ be a field of an arbitrary characteristic and
$\T: \mpp^1_{\kk}\times\mpp^n_{\kk} \rightarrow \mpp^1_{\kk}\times\mpp^n_{\kk}$ a rational dominant map defined as in~(\ref{defT2}),
by the homogeneous polynomials $A_i'$, $i=0,1$ in $\ul{X}'$ of degree $r$, and polynomials $A_i$, $i=0,\dots,n$ bi-homogeneous in $\ul{X}'$ and $\ul{X}$,
of bi-degree $(s,q)$.

Let $f_1(\b{z})$,...,$f_n(\b{z}) \in \kk[[\b{z}]]$
and $n_1\in\{1,\dots,n\}$, $C_1\in\mrr^+$.
We denote, as before, $\ul{f}=(1,\b{z},1,f_1(\b{z}),...,f_n(\b{z}))$.

Suppose moreover that there exists $\lambda \in \mrr_{>0}$, such that for all $Q\in\AnneauDePolynomes$
\begin{equation} \label{section_AB_IElambda}
\Ordz Q(\T(\ullt{f})) \geq \lambda \Ordz Q(\ullt{f}),
\end{equation}
and that there exists a
constant $K_0 \in \mrr^{+}$ (dependent on $\T$ and
$\ullt{f}$ only) such that for every positive integer
\begin{equation}\label{Nmajoration}
    N\leq C_m
\end{equation}
(where the constant $C_m$ is introduced in Definition~\ref{def_Cm})
every irreducible ${\T}^N$-stable variety $W\varsubsetneq\mpp^1\times\mpp^n$ (defined over the field $\kk$) of dimension $\dim W\leq n-n_1+1$
satisfies necessarily
\begin{equation} \label{RelMinN}
\ord_{\bt{f}}(W) < K_0\left(\dd_{(0,\dim W)}W+\dd_{(1,\dim W-1)}W\right).
\end{equation}
Then there exists a constant $K_1>0$ such that for all $P \in \AnneauDePolynomes\setminus\idp_{\ul{f}}$ satisfying for all $C\geq C_1$
\begin{equation}\label{condition_n1F}
    i_0(\Z_C(P))\geq n_1,
\end{equation}
satisfies also
\begin{equation} \label{LdMpolynome}
\ordz(P(\ullt{f})) \leq K_1(\deg_{\ul{X'}}P + \deg_{\ul{X}}P + 1)(\deg_{\ul{X}} P + 1)^n.
\end{equation}
\end{theorem}

In case of linear system we can provide an unconditional result. The proof is the same as that of Theorem~3.1 in~\cite{EZ2011} (or Theorem~3.11 in~\cite{EZ2010}).

\begin{theorem} \label{theoNishioka}
Let $\kk$ be a field of an arbitrary characteristic and $\T:\mpp^1_{\kk}\times\mpp^n_{\kk}\rightarrow\mpp^1_{\kk}\times\mpp^n_{\kk}$ a map defined by~(\ref{defT2}) with the polynomials $A_i$ linear in $\ul{X}$. Assume that
\begin{equation}\label{theoNishioka_lambda_pgq2}
    \lambda:=\ordz p(\b{z})\geq 2.
\end{equation}
Suppose that there is a solution $\ul{f}=(1,f_1(\b{z}),\dots,f_n(\b{z}))$ of the system of functional equations~(\ref{relsTopfer}) associated to $\T$. Denote by $t_{\ul{f}}$ the transcendence degree of $\ul{f}$ over $\kk(z)$ (see Definition~\ref{def_tf}).

Then there exists a constant $K_1$ such that for any polynomial $P\in\AnneauDePolynomes\setminus\idp_{\ul{f}}$ one has
\begin{equation*}
    \ordz(P(\ullt{f})) \leq K_1(\deg_{\ul{X'}}P + \deg_{\ul{X}}P + 1)(\deg_{\ul{X}} P + 1)^{t_{\ul{f}}}.
\end{equation*}
\end{theorem}


\subsubsection*{Multiplicity estimates for solutions
of algebraic differential equations} \label{subsection_ApplicationsDifferential}

In this subsection we consider an $n$-tuple $\ull{f}=(f_1(\b{z}),\dots,f_n(\b{z}))$ of analytic functions (or, more generally, power series) satisfying a system of differential equations
\begin{equation} \label{syst_diff}
f_i'(\b{z})=\frac{A_i(\b{z},\ull{f})}{A_0(\b{z},\ull{f})}, \quad i=1,\dots,n,
\end{equation}
where $A_i(\b{z},X_1,\dots,X_n)\in\kk[\b{z},X_1,\dots,X_n]$ for $i=0,...,n$ (we suppose that $A_0$ is a non-zero polynomial).

We associate to the system~(\ref{syst_diff}) the following derivation
\begin{equation} \label{defD}
D = A_0(\b{z}, X_1,\dots, X_n)\diff{\b{z}} + \sum_{i=1}^nA_i(\b{z}, X_1,\dots, X_n)\diff{X_i}.
\end{equation}
This operator is an application $D:\kk[\b{z},X_1,\dots,X_n]\rightarrow\kk[\b{z},X_1,\dots,X_n]$. We also consider $D$ as acting on $\AnneauDePolynomes=\kk[X_0',X_1'][X_1,\dots,X_n]$ defining
\begin{equation} \label{defhD}
D = {^h\!A}_0(X_0',X_1', X_1,\dots, X_n)\diff{X_1'}+\sum_{i=1}^n{^h
\!A_i}(X_0',X_1', X_1,\dots, X_n)\diff{X_i},
\end{equation}
where $^h\!P$ denotes the bi-homogenization of the polynomial $P\in
\kk[\b{z},X_1,\dots,X_n]$:
\begin{equation*}
  ^h\!P(X_0',X_1', X_1,\dots, X_n):=X_0'^{\deg_{\b{z}}P}\cdot
X_0^{\deg_{\ul{X}}P}\cdot P\left(\frac{X_1'}{X_0'},\frac{X_1}{X_0},
\dots,\frac{X_n}{X_0}\right).
\end{equation*}
One readily verifies~$D({^h\!P})={^h\!\left(D(P)
\right)}$, so the application $D:\AnneauDePolynomes\rightarrow
\AnneauDePolynomes$ is exactly the "bi-homogenization" of $D:
\kk[\b{z},X_1,\dots,X_n]\rightarrow\kk[\b{z},X_1,\dots,X_n]$.

The application $D$ is a correct application with respect to any
ideal $\idp \subset \AnneauDePolynomes$, according to Corollary~2.11 of~\cite{EZ2011} (or also corollary~2.10 of~\cite{EZ2010}).

Let's deduce from Theorem~\ref{LMGP} an improvement of the following Nesterenko's famous theorem (proved in~\cite{N1996}):

\begin{theorem}[Nesterenko, see Theorem~1.1 of Chapter~10, \cite{NP}] \label{theoNesterenko_classique}
Suppose that functions
\begin{equation*}
\ull{f} = (f_1(\b{z}),\dots,f_n(\b{z})) \in \mcc[[\b{z}]]^n
\end{equation*}
are analytic at the point $\b{z}=0$ and form a solution of the system~(\ref{syst_diff}) with $\kk=\mcc$.
If there exists a constant $K_0$ such that every $D$-stable prime ideal $\idp \subset \mcc[X_1',X_1,\dots,X_n]$,
$\idp\ne(0)$, satisfies
\begin{equation} \label{ordIleqKdegI}
\min_{P \in \idp}\ordz P(\b{z},\ull{f}) \leq K_0,
\end{equation}
then there exists a constant $K_1>0$ such that for any polynomial $P \in
\mcc[X_1',X_1,\dots,X_n]$, $P\ne 0$, the following inequality holds
\begin{equation} 
\ordz(P(\b{z},\ull{f})) \leq K_1(\deg_{\ul{X}'} P + 1)(\deg_{\ul{X}} P + 1)^n.
\end{equation}
\end{theorem}

\begin{remark}
Condition~\eqref{ordIleqKdegI} is the \emph{$D$-property}~\cite{N1996}.
Assuming $A_0(0,\ull{f}(0))\ne 0$ in the system~(\ref{syst_diff}), it is easy to verify the condition~(\ref{ordIleqKdegI}), cf.~\cite{NP}, chapitre 10, example~1 (p.~150).
Also, the condition~(\ref{ordIleqKdegI}) is established in the case when the polynomials $A_i$, $i=0,\dots,n$ are of degree 1 in $X_1,\dots,X_n$, cf.~\cite{N1974}.
In the latter case the proof is based on the differential Galois theory.
\end{remark}

In what follows we denote by $\cK_{prime}$ the class of prime ideals of $\A$ and $\cK_{primary}$ the class of primary ideals of $\A$. Using Theorem~\ref{LMGP} we can replace~(\ref{ordIleqKdegI}) in Theorem~\ref{theoNesterenko_classique} by
a weaker assumption, notably a $\left(D,\cK_{prime}\right)$-property (see Definition~\ref{def_weakDproperty}). In the same time we provide a result valid in an arbitrary characteristic.

\begin{theorem} \label{LMGPD}
Let
$(f_1(\b{z}),\dots,f_n(\b{z})) \in \kk[[\b{z}]]^n$
be a set of formal power series forming a solution of the system~(\ref{syst_diff}). 
We assume that $\ul{f}=(1:z,1:f_1(\b{z}):\dots:f_n(\b{z})$ satisfies the $\left(\phi,\cK_{primary}\right)$-property, and if $\car\kk=0$ we assume only that $\ul{f}$ satisfies the $\left(\phi,\cK_{prime}\right)$-property.
Under these conditions there is a constant $K>0$ such that every $P \in
\AnneauDePolynomes \setminus \idp_{\ul{f}}$ 
satisfies 
\begin{equation} 
\ordz(P(\b{z},\ull{f})) \leq K(\deg_{\ul{X'}} P + 1)(\deg_{\ul{X}} P + 1)^{t_{\ul{f}}}.
\end{equation}
\end{theorem}

\begin{center}%
         {\bfseries Acknowledgement\vspace{-.5em}}%
\end{center}%
\thanks{
The author would like to express his profound gratitude to Patrice \textsc{Philippon}. His interventions at many stages of  this research was of decisive importance.}

{\small



\def\cprime{$'$} \def\cprime{$'$} \def\cprime{$'$} \def\cprime{$'$}
  \def\cprime{$'$} \def\cprime{$'$} \def\cprime{$'$} \def\cprime{$'$}
  \def\cprime{$'$} \def\cprime{$'$} \def\cprime{$'$} \def\cprime{$'$}
  \def\cprime{$'$} \def\cprime{$'$} \def\cprime{$'$} \def\cprime{$'$}
  \def\polhk#1{\setbox0=\hbox{#1}{\ooalign{\hidewidth
  \lower1.5ex\hbox{`}\hidewidth\crcr\unhbox0}}} \def\cprime{$'$}
  \def\cprime{$'$}

\bigskip


\noindent{\footnotesize EZ\,: }\begin{minipage}[t]{0.9\textwidth}
\footnotesize{\sc University of York}\\
{\it E-mail address}\,:~~ \verb|evgeniy.zorin@york.ac.uk|\quad\emph{or}\quad\verb|EvgeniyZorin@yandex.ru|
\end{minipage}

}

\end{document}